\documentclass[a4paper,10pt]{amsart}

\usepackage{inputenc}
\usepackage{color}
\usepackage{graphicx}
\usepackage{times,mathptm}
\usepackage{amsfonts,amscd,amssymb,amsmath}
\usepackage{xypic}
\usepackage{latexsym}
\usepackage{tikz}
\usetikzlibrary{arrows,snakes,backgrounds}
\usetikzlibrary{patterns,arrows,decorations.pathreplacing}
  
\newtheorem{proposition}{Proposition}[section]
\newtheorem{lemma}[proposition]{Lemma}
\newtheorem{corollary}[proposition]{Corollary}
\newtheorem{theorem}[proposition]{Theorem}

\theoremstyle{definition}
\newtheorem{definition}[proposition]{Definition}
\newtheorem{example}[proposition]{Example}

\theoremstyle{remark}
\newtheorem{remark}[proposition]{Remark}

\newcommand{\proplabel}[1]{\label{prop:#1}}
\newcommand{\propref}[1]{Proposition~\ref{prop:#1}}

\newcommand{\lemlabel}[1]{\label{lem:#1}}
\newcommand{\lemref}[1]{Lemma~\ref{lem:#1}}

\newcommand{\thelabel}[1]{\label{the:#1}}
\newcommand{\theref}[1]{Theorem~\ref{the:#1}}

\newcommand{\corlabel}[1]{\label{cor:#1}}

\newcommand{\remlabel}[1]{\label{rem:#1}}
\newcommand{\remref}[1]{Remark~\ref{rem:#1}}

\newcommand{\deflabel}[1]{\label{def:#1}}
\newcommand{\defref}[1]{Definition~\ref{def:#1}}

\newcommand{\exalabel}[1]{\label{ex:#1}}
\newcommand{\exaref}[1]{Example~\ref{ex:#1}}

\newcommand\Hom{{\rm Hom}}

\newcommand\Ker{{\rm Ker}}

\newcommand\Mod{{\rm Mod}}

\newcommand\ot{\otimes}
\newcommand\mc{\mathcal}

\newcommand\ov{\overline}

\newcommand\rk{\rm rk}

\newcommand\mf{\mathfrak}
\newcommand\ad{\rm ad}
\newcommand\End{\rm End}
\newcommand\Tr{\rm Tr}
\newcommand\wt{\widetilde}

\newcommand\liesl{\mathfrak{sl}}
\newcommand\so{\mathfrak{so}}
\newcommand\liesp{\mathfrak{sp}}

\newcommand\Aut{\rm Aut}
\newcommand\Int{\rm Int}

\def\hookmapright#1{\smash{\mathop{\hookrightarrow}\limits^{#1}}}

\makeindex
\begin{document}
\title{Glider representations of chains of semisimple Lie algebras}
\author[F. Caenepeel]{Frederik Caenepeel}
\address{Department of Mathematics, University of Antwerp, Antwerp, Belgium}
\email{Frederik.Caenepeel@uantwerpen.be}
\begin{abstract}
We start the study of glider representations in the setting of semisimple Lie algebras. A glider representation is defined for some positively filtered ring $FR$ and here we consider the right bounded algebra filtration $FU(\mf{g})$ on the universal enveloping algebra $U(\mf{g})$ of some semisimple Lie algebra $\mf{g}$ given by a fixed chain of semisimple Lie subalgebras $\mf{g}_1 \subset \mf{g}_2 \subset \ldots \subset \mf{g}_n = \mf{g}$. Inspired by the classical representation theory, we introduce so called Verma glider representations. Their existence is related to the relations between the root systems of the appearing Lie algebras $\mf{g}_i$. In particular, we consider chains of simple Lie algebras of the same type $A,B,C$ and $D$.
\end{abstract}
\thanks{The author is Aspirant PhD Fellow of FWO}
\maketitle
\section{introduction}
The notion of a glider representation appeared for the first time in \cite{NVo1} and can be considered as a generalization of a module. They are defined over some positively filtered ring $FR$. If $S = F_0R$ is the subring given by the filtration, then a glider representation is an $S$-submodule $M$ of an $R$-module $\Omega$ together with a descending chain $M = M_0 \supset \ldots \supset M_n \supset \ldots$ such that $F_mRM_n \subset M_{n-m}$ for $m \leq n$. In fact, a glider representation is a special case of a fragment, also introduced in \cite{NVo1}. In the introduction of \cite{CVo1}, the authors describe multiple situations where glider representations offer a new viewpoint. In loc. cit. one studies glider theory for standard filtrations, i.e. $F_nR = (F_1R)^n$ for $n \geq 1$, appearing in (non-)commutative algebraic geometry. In \cite{CVo},\cite{CVo2} one considers finite semisimple Artinian algebra filtrations on the group algebra $KG$ of some finite group, given by a chain of (normal) subgroups $e < G_1 < \ldots < G_n = G$. For both the geometric and group theoretic situation, the glider theory reveals new information and raises new questions on the existing theories.\\
   
In this paper we enter the world of Lie algebras with this new machinery. Concretely, we consider chains of semisimple Lie algebras $\mf{g}_1 \subset \mf{g}_2 \subset \ldots \subset \mf{g}_n$, which yield finite algebra filtrations on the universal enveloping algebra $U(\mf{g}_n)$ by putting $F_iU(\mf{g}_n) = U(\mf{g}_{i+1})$ for $i =0, \ldots, n-2$, $F_mU(\mf{g}_n) = U(\mf{g}_n)$ for $m \geq n-1$. When dealing with Lie algebras one uses the beautiful geometry of the root systems, which appear by considering some Cartan subalgebra $\mf{h}$. For a chain of Lie algebras one can fix a chain of such Cartan subalgebras $\mf{h}_1 \subset \mf{h}_2 \subset \ldots \subset \mf{h}_n$. Elements of the root system live in the dual space $\mf{h}^*$ and the above chain of Cartan subalgebras yields a sequence of projections
$$\xymatrix{ \mf{h}_n^* \ar@{->>}[r] &  \mf{h}_{n-1}^*  \ar@{->>}[r] & \ldots  \ar@{->>}[r] & \mf{h}_1^*.}$$
We want to generalize the notion of a Verma module to a Verma glider. After fixing some Cartan subalgebra $\mf{h}$, Verma modules $M(\lambda)$ are indexed by functionals $\lambda \in \mf{h}^*$, so when fixing a chain of Cartan subalgebras $\mf{h}_1 \subset \ldots \subset \mf{h}_n$ as before, we would like to make a connection between the different labeling sets $\mf{h}^*$. To do this, we derive a condition on the embeddings $\mf{g}_i \subset \mf{g}_{i+1}$ appearing in the chain of semisimple Lie algebras, such that the inclusion $\iota: U(\mf{g}_i) \hookmapright{} U(\mf{g}_{i+1})$ of the universal enveloping algebras behaves nicely. By this, we mean that there exists a choice of bases $\Delta(j)$ of the root systems $\Phi_j$ of $\mf{g}_j$ ($j=i,i+1$) such that $\iota(U(\mf{n}_i)) \subset U(\mf{n}_{i+1})$. In here, $\mf{n}$ denotes the nilpositive part of the Lie algebra $\mf{g} = \mf{n}^- \oplus \mf{h} \oplus \mf{n}$ determined by the Cartan subalgebra $\mf{h}$ (it is the subvector space spanned by the eigenvectors of the positive roots).  It is exactly this behavior that makes it possible to relate Verma modules for the different Lie algebras $\mf{g}_i$. We devote a first section to determining the right condition and we provide both examples and counterexamples.\\

In the next section we recall the definition of a glider representation and briefly review some results obtained in \cite{CVo}, \cite{EVO}. For chains of semisimple Lie algebras satisfying the condition from the previous section, we define Verma gliders to be special glider representations $\Omega \supset M \supset \ldots \supset M_i \supset \ldots$ with regard to the positive algebra filtration on the universal enveloping algebra $U(\mf{g}_n)$ given by the chain. Starting from functionals $\lambda_i \in \mf{h}_i^*$ for $i=1,\ldots,n$ we explain how to construct such Verma gliders. The idea is to embed a $\mf{g}_i$-Verma module inside a $\mf{g}_{i+1}$-Verma module, the inclusion being an embedding of $U(\mf{g}_i)$-modules. This leads to the quest for elements $z \in U(\mf{n}^-_{i+1})$ such that $U(\mf{n}_i)\cdot zv^+ = 0$, where $v^+$ is the highest weight vector of a $\mf{g}_{i+1}$-Verma module. If the element $z$ is an eigenvector in $\mf{g}_{i+1}$, we call such an element an embedding element for $\mf{g}_{i}$ in $\mf{g}_{i+1}$. Imposing a mild additional condition on the inclusions of Lie algebras, we derive that the embedding elements all lie in the centralizer $C_{\mf{n}_i}(\mf{n}_{i+1}^-) = \{ z \in \mf{n}_{i+1}^-~\vline~ [\mf{n}_i,z] = 0\}$ of $\mf{n}_i$ inside $\mf{n}_{i+1}$. Of course, there is no harm in choosing $z \in U(\mf{n}_{i+1}^-)$ and it follows that any element in the centralizer $C_{\mf{n}_i}(U(\mf{n}_{i+1}^-))$ satisfies $U(\mf{n}_i)\cdot zv^+ = 0$. We prove that a PBW-monomial $z= y_{\alpha_1}^{r_1}\ldots y_{\alpha_m}^{r_m}$ lies in the centralizer $C_{\mf{n}_i}(U(\mf{n}_{i+1}^-))$ if and only if all appearing $y_{\alpha_i}$ are embedding elements (the notation is explained in section 2). However, it appears not to be true, that every element in the centralizer $C_{\mf{n}_i}(U(\mf{n}_{i+1}^-))$ is generated by the embedding elements, see \exaref{fail}.\\

In section 4 we study irreducibility of Verma gliders, in particular for chains of simple Lie algebras of the same type $A,B$ or $D$. It is known that ordinary Verma modules $M(\lambda)$ are irreducibly exactly when $\lambda \in \Lambda^+$ is dominant integral. We are able to extend this result for Verma gliders, see \theref{gene}. Finally, in section 5 we answer a question that arose by looking at the glider theory for chains of Lie algebras. In constructing Verma gliders, we introduced the so called embedding elements for an inclusion $\mf{g}_1 \subset \mf{g}_2$ of Lie algebras. These particular eigenvectors are nilpotent elements of $\mf{g}_2$, hence lie in some nilpotent orbit. We ask ourselves which nilpotent orbits we reach by only considering the embedding elements and all their linear combinations. We give an answer to this question for Lie algebras $\mf{g}_1 \subset \mf{g}_2$ of resp. rank $n<m$ and of the same type $A,B,C$ or $D$.\\

\section{Inclusions of semisimple Lie algebras}
Throughout we work with finite dimensional semisimple Lie algebras $\mf{g}$ with Lie bracket $[-,-]$ and the ground field $K$ is assumed to be algebraically closed of characteristic 0. A Lie subalgebra $\mf{h} \subset \mf{g}$ is a subvector space such that $[\mf{h},\mf{h}] \subset \mf{h}$. If such a subalgebra $\mf{h}$ consists entirely of semisimple elements, that is, elements for which $\ad_x \in \End(\mf{g})$ is diagonalizable , we call $\mf{h}$ toral. A toral subalgebra is abelian and if  $\mf{h}$ is a maximal toral subalgebra, then the centralizer $C_\mf{g}(\mf{h})$ equals $\mf{h}$. A maximal toral subalgebra is also referred to as a Cartan subalgebra.\\

Fix some Cartan subalgebra $\mf{h} \subset \mf{g}$. Since $\mf{h}$ is abelian, $\{ \ad_h, h \in \mf{h}\}$ is a commuting family of semisimple endomorphisms of $\mf{g}$. Hence we can simultaneously diagonalize this family and obtain a decomposition 
$$ \mf{g} = \mf{h} \oplus \bigoplus_{\alpha \in \Phi} \mf{g}_\alpha,$$
where $\mf{g}_\alpha = \{ x \in \mf{g}~\vline~ \forall h \in \mf{h}: [h,x] = \alpha(h)x\}$ and $\Phi \subset \mf{h}^*$. Of course, $\mf{h} = \mf{g}_0$ is the eigenspace of $\mf{g}$ with eigenvector $0$. We call the elements of $\Phi$ the roots of $\mf{\alpha}$, $\Phi$ is called the root system of $\mf{g}$ and the above decomposition is termed a root space decomposition. Although the Cartan subalgebra $\mf{h}$ is not unique, the root system is uniquely determined by the Lie algebra. In fact, more is true, the root system can be divided into two subsets $\Phi = \Phi^+ \cup \Phi^-$, where $\Phi^+$ denotes the set of positive roots, $\Phi^-$ the set of negative roots and such that $\Phi^- = \{ -\alpha~\vline~ \alpha \in \Phi^+\}$. For any positive root $\alpha \in \Phi^+$, there exist elements $x_\alpha \in \mf{g}_\alpha, y_\alpha \in \mf{g}_{-\alpha}, h_{\alpha} \in \mf{h}$ such that the vector space generated by $\{h_\alpha,x_\alpha,y_\alpha\}$ is a Lie algebra isomorphic to $\liesl_2$. We denote this Lie algebra by $\liesl_{\alpha}$.\\

The geometry of the root system is determined by a symmetric bilinear form $\kappa$, defined by $\kappa(x,y) = \Tr(\ad_x \circ \ad_y)$. This $\kappa$ is called the Killing form of $\mf{g}$ and it is nondegenerate if and only if $\mf{g}$ is semisimple. In this case, the Killing form restricted to $\mf{h}$ is also nondegenerate and by transferring the inner product, $\mf{h}^*$ becomes Euclidean. We denote the inner product on $\mf{h}^*$ by $\langle-,-\rangle$. We refer to \cite[Chapter 8]{Hu} for details about the geometric properties of the root space $\Phi$. A subset $\Delta \subset \Phi$ is called a base if $\Delta$ is a base for $E = \mf{h}^*$ and if each root $\alpha \in \Phi$ can be written as a linear combination of the elements of $\Delta$ with integral coefficients, which moreover are all nonnegative or nonpositive, thus leading to the decomposition $\Phi = \Phi^+ \cup \Phi^-$. The elements of such a base are called simple roots. Denote by $P_\alpha$ the hyperplane in $E$ perpendicular to $\alpha \in E$. Let $\gamma \in E \setminus \bigcup_{\alpha \in \Phi} P_\alpha$, then the set $\Phi^+(\gamma) = \{ \alpha \in \Phi, ~ \langle\gamma, \alpha\rangle > 0 \}$ consists of all roots lying on the same side of the hyperplane orthogonal to $\gamma$. It can be shown that the subset $\Delta(\gamma) \subset \Phi^+(\gamma)$ of all indecomposable vectors in $\Phi^+(\gamma)$ forms a base for $\Phi$ and that every base $\Delta$ for $\Phi$ is of the form $\Delta(\gamma)$ for some $\gamma \in E \setminus \bigcup_{\alpha \in \Phi} P_\alpha$. In fact, this shows that the number of bases is in one-to-one correspondence with the number of Weyl chambers of $E$. Elements $\gamma \in E \setminus \bigcup_{\alpha \in \Phi} P_\alpha$ are called regular.\\

If one fixes a base $\Delta = \Delta(\gamma) = \{\alpha_1,\ldots, \alpha_n\}$ for some suitable $\gamma \in E = \mathbb{R}^{n}$, where $n= \dim(\mf{h}) = {\rm rk}(\mf{g})$, then we can write $\mf{g} = \mf{n}^- \oplus \mf{h} \oplus \mf{n}$, with $\mf{n} = \bigoplus_{\alpha \in \Phi^+} \mf{g}_\alpha$. The Poincar\'e-Birkhoff-Witt theorem states that there is an isomorphism $U(\mf{g}) \cong U(\mf{n}^-) \ot U(\mf{b})$, where $\mf{b} = \mf{h} \oplus \mf{n}$. To calculate in the universal enveloping algebra, we fix an ordering of the positive roots $\alpha_1,\ldots, \alpha_m$ and pick elements $h_i = h_{\alpha_i}, x_i = x_{\alpha_i}, y_i = y_{\alpha_i}$ generating the $\liesl_{\alpha_i}$. Then the elements of the form 
$$y_1^{r_1}\ldots y_m^{r_m}h_1^{s_1}\ldots h_n^{s_n}x_1^{t_1}\ldots x_m^{t_m}, \quad r_i,s_i,t_i \in \mathbb{N}$$
form a base for $U(\mf{g})$. We call them PBW-monomials. Such an element belongs to the weight space $U(\mf{g})_\omega$, where $\omega = \sum_{i=1}^m (t_i -r_i)\alpha_i \in \mf{h}^*$ and where $U(\mf{g})_\omega = \{ z \in U(\mf{g})~\vline~ h \cdot z = \omega(h)z {\rm~for~all~} h \in \mf{h}\}$. To any $\lambda \in \mf{h}^*$ we can associate a left $U(\mf{g})$-module $M(\lambda)$ as follows: since $\mf{h} \cong \mf{b}/\mf{n}$ as Lie algebras, $\lambda$ yields a 1-dimensional $\mf{b}$-module $\mathbb{C}_\lambda$ with trivial $\mf{n}$-action. By the PBW theorem, $U(\mf{g})$ has a canonical $(U(\mf{g}),U(\mf{b}))$-bimodule structure. Define $M(\lambda) := U(\mf{g}) \ot_{U(\mf{b})} \mathbb{C}_\lambda$, which we call the Verma module associated to $\lambda$.  Observe that $M(\lambda) \cong U(\mf{n}^-) \ot \mathbb{C}_\lambda$, which is a free $U(\mf{n}^-)$-module of rank 1. It is easy to see that $M(\lambda)$ is a highest weight module with maximal vector $v_\lambda^+ = 1 \ot 1$ of weight $\lambda$. Moreover, the set of weights is $\lambda - \Gamma$, where $\Gamma$ is the set of all $\mathbb{Z}^+$-linear combinations of simple roots. In other words, the action highly depends on the structure of the root system.\\

As stated in the introduction, we would like to study glider representations for a chain of universal enveloping algebras $U(\mf{g}_1) \subset U(\mf{g_2}) \subset \ldots \subset U(\mf{g_n})$ associated to a chain of semisimple Lie algebras $\mf{g}_1 \subset \mf{g}_2 \subset \ldots \subset \mf{g}_n$. More specifically, we want to study so called Verma gliders, the definition of which is postponed to section 3. Nevertheless, we already mention here that our study of Verma gliders naturally leads to embeddings of $\mf{g}_{i}$-Verma modules as $U(\mf{g}_1)$-modules inside  $\mf{g}_{i+1}$-Verma modules. By the structure of $U(\mf{g})$ we know that the elements $y_1^{r_1}\ldots y_m^{r_m}$ form a base of $U(\mf{n}^-)$ hence we must make some logical connection between roots of $\mf{g}_i$ and roots of $\mf{g}_{i+1}$. In doing so, we are forced to put some condition on the inclusion of Lie algebras, as we explain now.\\

Consider an embedding $\iota: \mf{g}_1 \hookmapright{} \mf{g}_2$ of semisimple Lie algebras and choose some Cartan subalgebra $\mf{h}_1$ in $\mf{g}_1$. One can extend $\mf{h}_1$ to a Cartan subalgebra $\mf{h}_2$ of $\mf{g}_2$, so we obtain the following root space decompositions
$$\mf{g}_1 = \mf{h}_1 \oplus \bigoplus_{\alpha \in \Phi_1} \mf{g}_{1,\alpha}, \quad \mf{g}_2 = \mf{h}_2 \oplus \bigoplus_{\beta \in \Phi_2} \mf{g}_{2,\beta}.$$
The root systems $\Phi_1$ and $\Phi_2$ are subsets of $\mf{h}_1^*$, $\mf{h}_2^*$ respectively, and the inclusion $\mf{h}_1 \subset \mf{h}_2$ gives rise to a projection $\mf{h}_2^* \to \mf{h}_1^*$ given by restriction. We want the following condition to hold
\begin{equation}\label{condition}\forall \alpha \in \Phi_1:~ \#\{ \beta \in \Phi_2 ~\vline~ \pi(\beta) = \alpha\} = 1.\end{equation}

Suppose that this condition holds and let $\alpha \in \Phi_1$. If $\beta$ is the unique root in $\Phi_2$ such that $\pi(\beta) = \alpha$, then $x^1_\alpha \in \mf{g}_{2,\beta}$, which entails that $\iota(x^1_\alpha) = c_\beta x^2_\beta$, with $c_\beta \in K$ (the upper index $i$ refers to the Lie algebra $\mf{g}_i$ we are working in). In fact, there is some liberty in choosing the element $x^2_\beta$, so one may assume that $c_\beta = 1$. Here we did not fix a base yet, so we did not speak about positive or negative roots. If bases are fixed, then you have to adjust the notations. E.g. if $\alpha$ is negative and $\beta$ is positive, then this means that $\iota(y^1_{-\alpha}) =x^2_\beta$. The condition allows to appoint a single root of the bigger Lie algebra to any root of the smaller one. We denote by $\Phi_1^*$ the set of roots obtained in this way. We have an equivalent characterizations in terms of eigenvectors.
\begin{proposition}\proplabel{equivalent}
Let $\Delta(i)$ be a base of $\Phi_i$ for $i=1,2$. Condition \eqref{condition} is equivalent to saying that for all $\alpha \in \Phi_1^+, x^1_\alpha$ is an eigenvector in $\mf{g}_2$ with regard to $\mf{h}_2$ such that if $x^1_\alpha = z^2_\beta$ ($z^2_\beta = x^2_\beta$ or $y^2_\beta$, $\beta \in \Phi_2^+$) then $h^2_\beta \in \mf{h}_1$.
\end{proposition}
\begin{proof}
We already showed that if \eqref{condition} holds, that the $x^1_\alpha$ are eigenvectors. If $x^1_\alpha = \lambda x^2_\beta$, then $y^1_\alpha = \mu y^2_\beta$, from which it follows that $h^2_\beta =[x^2_\beta,y^2_\beta] = \frac{1}{\lambda \mu} [x^1_\alpha, y^1_\alpha] =  \frac{1}{\lambda \mu}h^1_\alpha \in \mf{h}_1$. The case $x^1_\alpha = \lambda y^2_\beta$ is analogous. Conversely, suppose that $\beta, \gamma \in \Phi_2$ are such that $\pi(\beta) = \pi(\gamma) = \alpha \in \Phi_1$. Without loss of generality, we may assume that $x^1_\alpha \in \mf{g}_{2,\beta}$, which implies that $h^2_\beta \in \mf{h}_1$. It follows that $\alpha(h^2_\beta) = \beta(h^2_\beta) = \gamma(h^2_\beta) = 2$. Hence 
$$2 = \gamma(h^2_\beta) = \langle \gamma, \beta^\vee \rangle = 2 \cos(\theta_{\gamma, \theta}) \frac{ ||\gamma||}{||\beta||},$$
where $\beta^\vee = 2\beta/\langle \beta, \beta \rangle$. By the geometry of root systems, we must have that $\frac{||\gamma||}{||\beta||} = 1$ and $\cos(\theta_{\gamma,\beta}) = 1$, hence $\gamma = \beta$. 
\end{proof}

To show that this condition is independent of the choice of Cartan subalgebras $\mf{h}_1 \subset \mf{h}_2$ we have to recall some facts concerning automorphisms of Lie algebras. There is a normal subgroup $E(\mf{g})$ of the automorphism group $\Aut(\mf{g})$ generated by all $\exp(\ad_x), x$ a strongly ad-nilpotent element. For semisimple Lie algebras it holds that $E(\mf{g}) = \Int(\mf{g})$, where the latter group is the subgroup of $\Aut(\mf{g})$ generated by all $\exp(\ad_x), x$ a nilpotent element. Apparently, any two Cartan subalgebras $\mf{h}, \mf{h}'$ of a Lie algebra $\mf{g}$ are conjugated by an element of $\sigma \in E(\mf{g})$, i.e. $\sigma(\mf{h}) = \mf{h}'$ and with regard to Lie subalgebras there is a nice functorial behavior. If $\mf{g}' \subset \mf{g}$ is a Lie subalgebra, we can look at the subgroup $E(\mf{g}; \mf{g}')$ of $E(\mf{g})$ generated by all $\exp(\ad^\mf{g}_x), x \in \mf{g}'$ strongly ad-nilpotent. By restricting the automorphisms of $E(\mf{g};\mf{g}')$ one obtains the group $E(\mf{g}')$. We refer the reader to \cite[Chapter 8]{Hu} for a detailed overview of these facts.\\

Let $\mf{g}$ be a Lie algebra with Cartan subalgebras $\mf{h}, \mf{h}'$. If $\sigma \in E(\mf{g})$ is such that $\sigma(\mf{h}) = \mf{h}'$, then the roots with regard to $\mf{h}'$ are exactly the functionals $\beta\sigma^{-1}: \mf{h}' \to K$ where $\beta: \mf{h} \to K$ is a root with regard to $\mf{h}$. Moreover, it holds that $\sigma(x_\beta) = x_{\beta\sigma^{-1}}$. 
\begin{proposition}
Condition \eqref{condition} is independent under the orbit $\mc{O}_{\mf{h}_1 \subset \mf{h}_2}$ for the action of $E(\mf{g}_2;\mf{g}_1)$.
\end{proposition}
\begin{proof}
Let $\wt{\sigma} \in E(\mf{g}_2;\mf{g}_1)$ and denote its restriction to $\mf{g}_1$ be $\sigma$. We define $\mf{h}'_1= \sigma(\mf{h}_i)$ and $\mf{h}'_2 = \wt{\sigma}(\mf{h}_2)$. We have the following commutative diagram
$$\xymatrix{\mf{h}_2^* \ar@{->>}[r]^\pi \ar[d]_{\wt{\sigma}^*} & \mf{h}_1^* \ar[d]^{\sigma^*} \\
\wt{\sigma}(h_2)^* \ar@{->>}[r]^\pi & \mf{h}^{'*}_1}$$
If $\beta, \gamma \in \Phi_2$ are such that $\pi(\beta) = \pi(\gamma) = \alpha \in \Phi_1$, then $\sigma^*(\alpha) = (\pi \circ \wt{\sigma}^*)(\beta) = (\pi \circ \wt{\sigma}^*)(\gamma)$, which shows that the roots $\beta\sigma^{-1}, \gamma\wt{\sigma}^{-1} \in \Phi'_2$ both restrict to the root $\alpha\sigma^{-1} \in \Phi'_1$.
\end{proof}

To show that \eqref{condition} is independent of the choice of inclusions $\mf{h}_1 \subset \mf{h}_2$ of Cartan subalgebras, it would suffice to show that the condition is independent for inclusions $\mf{h}_1 \subset \mf{h}_2$ and $\mf{h}_1 \subset \mf{h}_2'$. Imposing a condition on the ranks of the $\mf{g}_i$, this follows from.

\begin{proposition}
Assume that $\rk(\mf{g}_2) = \rk(\mf{g}_1) + 1$. Then \eqref{condition} is independent of the choice of Cartain subalgebras $\mf{h}_1 \subset \mf{h}_2$.
\end{proposition}
\begin{proof}
From the remark above, it suffices to show independence for inclusions $\mf{h}_1 \subset \mf{h}_2$ and $\mf{h}_1 \subset \mf{h}_2'$. Suppose that \eqref{condition} holds for $\mf{h}_1 \subset \mf{h}_2$ and let $\alpha \in \Phi_1$. We know there exists a unique $\beta \in \Phi_2$ such that $x^1_\alpha = \lambda x^2_\beta$ (by adjusting bases we may assume that both $\alpha$ and $\beta$ are positive within their respective root system and by rescaling we may assume that $\lambda = 1$). Take some $\tau \in E(\mf{g}_2)$ sending $\mf{h}_2$ to $\mf{h}_2'$. For $h \in \mf{h}_2 \cap \mf{h}_2'$, we have that $h = \tau(h')$ for some $h' \in \mf{h}_2 \cap \mf{h}_2'$. This shows that $\tau$ restricts to an automorphism of $\mf{h}_2 \cap \mf{h}_2'$ which equals $\mf{h}_1$ by the hypotheses on the ranks of $\mf{g}_1$ and $\mf{g}_2$. It follows that $h^2_{\beta\tau^{-1}} = \tau(h^2_\beta) = \tau(h^1_\alpha) \in \mf{h}_1$ and also that the root $\alpha\tau_{\vline \mf{h}_1}^{-1}$ with regard to $\tau^{-1}(\mf{h}_1) = \mf{h}_1$ has associated element $x^1_{\alpha\tau_{\vline \mf{h}_1}^{-1}} = \tau(x^1_\alpha) = \tau(x^2_\beta) = x^2_{\beta\tau^{-1}}$. We have shown that for all roots $\alpha' = \alpha\tau_{\vline \mf{h}_1}^{-1} \in \Phi_1$, the associated element $x^1_{\alpha'} = x^2_{\beta\tau^{-1}}$ is an eigenvector such that $h^2_{\beta\tau^{-1}} \in \mf{h}_1$. 
\end{proof}
\begin{remark}
Without imposing the condition on the ranks of the Lie algebras, the author is not able to prove this, neither can come up with a counterexample.
\end{remark}
Not all embeddings $\mf{g}_1 \subset \mf{g}_2$ satisfy condition \eqref{condition}.
\begin{example}\exalabel{slso}
Consider $\liesl_2 \subset \so_4 \cong \liesl_2 \times \liesl_2$, where the $\liesl_2$ is embedded diagonally. Choose the Cartan subalgebra $\mf{h}_2$ to be the subalgebra generated by the diagonal matrices $H_1 = E_{11} - E_{33}$ and $H_2 = E_{22} - E_{44}$. If $\{L_1,L_2\}$ denotes a dual base for $\{H_1,H_2\}$ in $\mf{h}_2^*$, then the positive roots of $\so_4$ are $L_1 - L_2$ and $L_1 + L_2$.  The copies of $\liesl_2$ for $L_1-L_2$, resp. $L_1 + L_2$ are generated by
$$\begin{array}{l}
 h_{L_1-L_2} = H_1 - H_2,\\
x_{L_1 - L_2} = E_{12} - E_{43},\\
 y_{L_1 - L_2} = E_{21} - E_{34},
 \end{array} {\rm~resp.~} 
 \begin{array}{c}
 h_{L_1 + L_2} = H_1 + H_2,\\
 x_{L_1 + L_2} = E_{14} - E_{23},\\
 y_{L_1 + L_2} = E_{32} - E_{41}.
 \end{array}$$
 We see that the nilpositive element of $\mf{g}_1 = \liesl_2$ is $x_\alpha = x_{L_1 - L_2} + x_{L_1 + L_2}$. The vector space generated by $2H_1 = (H_1 - H_2) + (H_1 + H_2)$ is a Cartan subalgebra $\mf{h}_1$ of $\mf{g}_1$. With regard to $\mf{h}_1$, the root $\alpha$ is $2L_1$ and we have that $\pi(L_1 - L_2) = \pi(L_1 + L_2)$, where $\pi: \mf{h}_2^* \to \mf{h}_1^*$ denotes the projection.
\end{example}
An element $x \in \mf{g}$ is nilpotent if $\ad_x \in \End(\mf{g})$ is nilpotent. We recall the Jacobson-Morozov theorem, \cite[Theorem 3.3.1]{CoMc}, which states that for any nonzero nilpotent element $x$, there exists a standard triple $\{h,x,y\}$. By this we mean that the subspace generated by these three elements is isomorphic to $\liesl_2$ with $\mf{h} = Kh$ a Cartan subalgebra, $x = x_\alpha$ and $y = y_\alpha$.  

\begin{lemma}\lemlabel{vectors}
Let $E$ be an $n$-dimensional Euclidean space with inner product $\langle -, - \rangle$ and let $\alpha_1, \ldots, \alpha_{n-1}$ be linearly independent vectors in $E$. Then there exists a hyperplane $P_\gamma, \gamma \in E$, such that all vectors $\alpha_i, 1 \leq i \leq n-1$ lie on the same side of the hyperplane.
\end{lemma}
\begin{proof}
We proceed by induction, the case $n=2$ being trivial. So assume the result holds in $n-1$-dimensional Euclidean spaces and take $n-1$ linearly independent vectors $\alpha_1,\ldots, \alpha_{n-1}$ in an $n$-dimensional space $E$. Let $\delta$ be perpendicular to all $\alpha_i, 1 \leq i \leq n-1$ and consider $\alpha_1, \ldots, \alpha_{n-2}$ inside $E' = <\alpha_1, \ldots, \alpha_{n-2}, \delta>$. By induction, we find some $\gamma' \in E'$ such that $\langle \gamma', \alpha_i\rangle > 0$ for all $1 \leq i \leq n-2$. If $\langle \gamma', \alpha_{n-1} \rangle > 0$, then $P_\gamma$ for $\gamma = \gamma'$ does the job. If $\langle \gamma', \alpha_{n-1} \rangle < 0$, we write 
$$<\alpha_1,\ldots,\alpha_{n-1}> = <\alpha_1, \ldots, \alpha_{n-2}> \oplus K\beta,$$
where $\beta \in E$ is a vector perpendicular to all $<\alpha_1,\ldots, \alpha_{n-2}>$, such that $\alpha_{n-1} = \sum_{i=1}^{n-1} c_i\alpha_i + \beta$. We have that $\langle \gamma' + \beta, \alpha_i \rangle > 0$ for all $1 \leq i \leq n-2$ and $\langle \gamma' + \beta, \alpha_{n-1}\rangle = \langle \gamma', \alpha_{n-1}\rangle + ||\beta ||^2$. By rescaling $\gamma'$ to $d\gamma'$ such that $d\langle \gamma', \alpha_{n-1} \rangle < || \beta||^2$, we  see that $\gamma = d\gamma' + \beta$ defines a suitable hyperplane. Finally, if $\langle \gamma', \alpha_{n-1} \rangle = 0$, then we choose $\gamma''$ close to $\gamma'$ such that $\langle \gamma'', \alpha_{n-1} \rangle > 0$ and $\langle \gamma'', \alpha_i \rangle >0$ for $1 \leq i \leq n-2$ still holds.
\end{proof}
\begin{proposition}\proplabel{nilpotent}
Let $\mf{g}$ be a semisimple Lie algebra of rank $> 2$ with root space decomposition $\mf{g} = \mf{n}^- \oplus \mf{h} \oplus \mf{n}$ with regard to some Cartan subalgebra $\mf{h}$. If $\alpha, \beta \in \Phi^+$, then the elements $x_\alpha + x_\beta$ and $x_\alpha + y_\beta$ are nilpotent.
\end{proposition}
\begin{proof}
The endomorphism $\ad_{x_\alpha+x_\beta}$ maps a weight vector to a linear combination of weight vectors with higher weight, so $x_\alpha + x_\beta$ is clearly nilpotent. By \lemref{vectors} we know there exists some $\gamma \in E = \mf{h}^*$ such that $\langle \gamma, \alpha \rangle >0$ and $\langle \gamma, -\beta \rangle >0$. We may assume that $\gamma$ is regular. With regard to the base $\Delta(\gamma)$ both $\alpha$ and $-\beta$ are positive roots, whence $x_\alpha + y_\beta = x^\gamma_\alpha + x^\gamma_{-\beta}$ is nilpotent, where $x^\gamma_\delta$ denotes the associated weight vector of a positive root $\delta$ with regard to $\Delta(\gamma)$.
\end{proof}
From the previous proposition we see that the inclusion of \exaref{slso} corresponds to the standard triple of the nilpotent element $x_\alpha + x_\beta$ for positive roots $\alpha, \beta$. In general, the inclusion of the associated triple $\{h,x,y\}$ of a nilpotent element of the form $x = x_{\alpha_1} + \ldots + x_{\alpha_m}$, where $\alpha_1, \ldots, \alpha_m$ are positive roots in a root system $\Phi$ of a semisimple Lie algebra $\mf{g}$ does not satisfy condition \eqref{condition}. It appears that this is the kind of behavior that causes condition \eqref{condition} to fail.

\begin{proposition}\proplabel{fail}
Let $\mf{g}_1 \subset \mf{g}_2$ be an embedding of semisimple Lie algebras with Cartan subalgebras $\mf{h}_1 \subset \mf{h}_2$. If condition \eqref{condition} does not hold and if $\rk(\mf{g}_1) = n < \rk(\mf{g}_2) = m$, then there exists a nilpotent element $x = x_{\alpha_1} + \ldots, x_{\alpha_r}$  with $r \leq n-m + 1$ and an associated triple $\{h, x,y\}$ with $h \in \mf{h}_1$ such that $\liesl_2 \cong <h,x,y> \subset \mf{g}_1$.
\end{proposition} 
\begin{proof}
The kernel of the projection $\pi: \mf{h}_2^* \to \mf{h}_1^*$ is $m-n$ dimensional, so there can exist at most $m -n +1$ different roots restricting to the same element in $\mf{h}_1^*$. Since condition \eqref{condition} is not satisfied, there exist roots $\alpha_1, \ldots, \alpha_r \in \Phi_2$ with $r \leq n-m+1$ such that $\pi(\alpha_1) = \ldots = \pi(\alpha_r) = \alpha \in \Phi_1$. By \lemref{vectors} there exists a base $\Delta$ such that the $\alpha_i, 1 \leq i \leq r$ are all positive. We have that $\mf{h}_1 \subset \bigcap_{i=1}^{r-1} \Ker(\alpha_i - \alpha_{i+1})$ and since $\alpha \in \Phi_1$ is a root, there is an associated weight vector $x^1_\alpha \in \mf{g}_1$ and some $h \in \mf{h}_1$ such that $\alpha(h) = 2$. The element $x^1_\alpha$ is not an eigenvector in $\mf{g}_2$, but lives in $\mf{g}_{2,\alpha_1} \oplus \ldots \oplus \mf{g}_{2,\alpha_r}$. We can rescale the elements $x^2_{\alpha_i}$ such that $x^1_{\alpha} = x^2_{\alpha_1} + \ldots + x^2_{\alpha_r}$ (the $y^2_{\alpha_i}$ then also have to be rescaled appropriately). Hence $[h, \sum_{i=1}^r x^2_{\alpha_i}] = 2\sum_{i=1}^r x^2_{\alpha_i}$ and it follows from the proof of \cite[Theorem 3.3.1]{CoMc} that $\{h, \sum_{i=1}^r x^2_{\alpha_i}, y\}$ is a triple. Because $\{h, x^1_\alpha, y^1_\alpha\}$ is also a triple in $\mf{g}_2$, we know that $y = y^1_\alpha \in \mf{g}_1$ (see \cite[Lemma 3.4.4]{CoMc}) which finishes the proof.
\end{proof}
For background on Levi subalgebras we refer to \cite[section 3.8]{CoMc}. We use the notations from loc. cit.

\begin{proposition}
Let $\mf{g}$ be a Lie algebra with fixed Cartan subalgebra $\mf{h}$. Let $\mf{l}_I \subset \mf{g}$ be the Levi subalgebra associated to a subset $I \subset \Delta$ and $\mf{g}_I = [\mf{l}_I, \mf{l}_I]$ the derived subalgebra.  The embedding $\mf{g}_I \subset \mf{g}$ satisfies condition \eqref{condition}.
\end{proposition}
\begin{proof}
As Cartan subalgebra of $\mf{g}_I$ we take $\mf{h}_I = \bigoplus_{\alpha \in I} Kh_{\alpha}$. Then $\mf{g}_I = \mf{n}_I^- \oplus \mf{h}_I \oplus \mf{n}_I$, with $\mf{n}_I = \bigoplus_{\alpha \in \Phi_I^+} \mf{g}_\alpha$. The result easily follows from \propref{equivalent}.
\end{proof}
\begin{example}\exalabel{Dynkin}
Embeddings $\mf{g}_1 \subset \mf{g}_2$ of simple Lie algebras of the same type $A,B,C$ or $D$ corresponding to a connected subdiagram of the associated Dynkin diagram all satisfy condition \eqref{condition}. To clarify what we mean, we give the example of the inclusion $\liesp_{2n} \subset \liesp_{2m}$ given by 

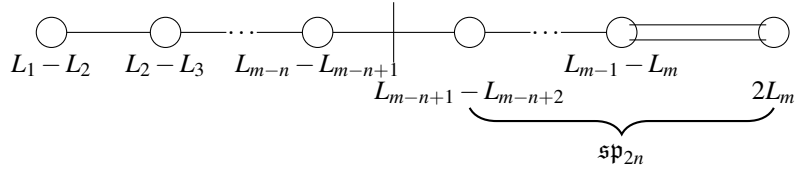
\begin{equation}
\begin{tikzpicture}\label{rootsystem}
 \draw (.5,0) circle (0.2cm) node[below=5pt] {$L_1 - L_2$};
  \draw ( 0.7,0)--( 2 -0.2,0);
 \draw (2,0) circle (0.2cm) node[below=5pt]{$L_2 - L_3$};
   \draw ( 2.2,0)--( 3 -0.2,0);
 \node at (3,0){$\ldots$};
 \draw ( 3.2,0)--(  4-0.2,0);
 \draw (4,0) circle (0.2cm) node[below=5pt] {$L_{m-n} - L_{m-n+1}$};
 \draw ( 4.2,0)--( 6 -0.2,0);
 \draw (6,0) circle (0.2cm) node[below=15pt] {$L_{m-n+1} - L_{m-n+2}$};
 \draw ( 6.2,0)--( 7 -0.2,0);
\node at (7,0) {$\ldots$};
\draw ( 7.2,0)--( 8 -0.2,0);
 \draw (8,0) circle (0.2cm) node[below=5pt] {$L_{m-1} - L_m$};
 \draw (8.1,0.1) -- (10 - 0.1, 0.1);
 \draw(8.1,-0.1) -- (10 -0.1, -0.1);
 \draw (10,0) circle (0.2cm) node[below=15pt] {$2L_m$};
 \draw (5,.4) -- (5,-.4);
 \draw [thick,decorate, decoration={brace,amplitude = 10pt,mirror}](6,-1) -- (10,-1) node[below=10pt,midway]{$\liesp_{2n}$};
\end{tikzpicture}
\end{equation} We use the notations from \cite{FuHa}. In other words, they just correspond to the derived Lie subalgebra of a suitable Levi subalgebra, and the claim follows from the previous proposition. A final word here on embeddings of type $A$: we only consider embeddings $\liesl_n \subset \liesl_m$ corresponding to a connected subdiagram of the $A_{m-1}$-diagram containing one of the end points.

\end{example}

Let $\mf{g}_1 \subset \mf{g}_2$ be an inclusion satisfying condition \eqref{condition} (with regard to $\mf{h}_1 \subset \mf{h}_2$) and with $\rk(\mf{g}_1) = n < \rk(\mf{g}_2) = m$. Pick a base $\Delta(1) = \{\alpha_1,\ldots, \alpha_n\}$ for $\mf{g}_1$ and denote by $\alpha_i^*, (i=1,\ldots,n)$ the unique root in $\Phi_2$ restricting to $\alpha_i$. \lemref{vectors} allows to pick a base $\Delta(2)$ such that all $\alpha_i^*$ are positive roots. This entails that the embedding $\iota: U(\mf{g}_1) \hookmapright{} U(\mf{g}_2)$ maps $U(\mf{n}_1)$ inside $U(\mf{n}_2)$. Indeed, if $\alpha_1 + \alpha_2$ is a positive root, then $x^1_{\alpha_1+\alpha_2}$ is an eigenvector in $\mf{g}_2$. Since
$$x^1_{\alpha_1 + \alpha_2} = \lambda [x^1_{\alpha_1}, x^1_{\alpha_2}] = \lambda' [x^2_{\alpha^*_1},x^2_{\alpha_2^*}] \subset \mf{g}_{2,\alpha^*_1+\alpha^*_2},$$
we see that $\alpha_1^* + \alpha^*_2 \in \Phi_2$ is the unique root restricting to $\alpha_1 + \alpha_2$. We end this section with 
\begin{proposition}
If $\alpha, \beta \in \Phi_1^+$ are such that $\alpha^* + \beta^* \in \Phi_2^+$, then $\alpha + \beta \in \Phi_1$ and $(\alpha + \beta)^* = \alpha^* + \beta^*$.
\end{proposition}\proplabel{extendedroots}
\begin{proof}
Straightforward, since $[x^1_\alpha, x^1_\beta] = [x^2_{\alpha^*},x^2_{\beta^*}] = \lambda x^2_{\alpha^* + \beta^*} \in \mf{g}_1$.
\end{proof}

 \section{Verma gliders}

In this section we will introduce so called Verma glider representations after recalling the general definition of a fragment and, in particular, of a glider representation introduced in \cite{NVo1} and refined in \cite{CVo}.

\begin{definition}\deflabel{fragment}
Let $FR$ be a positive filtered ring with subring $S = F_0R$. A (left) $FR$-fragment $M$ is a (left) $S$-module together with a descending chain of subgroups
$$M_0 = M \supseteq M_1 \supseteq \cdots \supseteq M_i \supseteq \cdots$$
satisfying the following properties\\
$\bf{f_1}$. For every $i \in \mathbb{N}$ there exists an $S$-module $M \supseteq M_i^* \supseteq M_i$ and there is given an operation of $F_iR$ on this $M_i^*$ by $\varphi_i: F_iR \times M_i^* \to M,~(\lambda,m) \mapsto \lambda.m$, satisfying $\lambda.(m + n) = \lambda.m + \lambda.n, 1.m = m, (\lambda + \delta).m = \lambda.m + \delta.m$ for $\lambda,\delta \in F_iR$ and $m,n \in M_i^*$.\\

$\bf{f_2}$. For every $i$ and $j \leq i$ we have a commutative diagram
$$\xymatrix{ M & M_{i-j} \ar@{_{(}->}[l]^i \ar@{^{(}->}[r]_i & M\\
F_iR \times M_i \ar[u]^{\varphi_i} & F_jR \times M_i \ar@{_{(}->}[l]^{i_F} \ar[u] \ar@{^{(}->}[r]_{i_M} & F_jR \times M_j \ar[u]_{\varphi_j}}$$

$\bf{f_3}$. For every $i,j,\mu$ such that $F_iRF_jR \subset F_\mu R$ we have $F_jRM_\mu \subset M_i^* \cap M_{\mu - j}$.
Moreover, the following diagram is commutative
$$\xymatrix{
F_iR \times F_jR \times M_\mu \ar[d]_{F_iR \times \varphi_\mu} \ar[rr]^{m \times M_\mu} && F_\mu R \times M_\mu \ar[d]_{\varphi_\mu} \\
F_iR \times M_{\mu - j} \ar[rr]^{\ov{\varphi_i}} && M},$$
in which $\ov{\varphi_i}$ stands for the action of $F_iR$ on $M_i^*$ and $m$ is the multiplication of $R$. Observe that the left vertical arrow is defined, since $1 \in F_0R$ implies that $F_jR \subset F_\mu R$. 
\end{definition}
For an $FR$-fragment structure on $M$, the chain $M \supseteq M_1^* \supseteq M_2^* \supseteq \cdots$ obviously also yields an $FR$-fragment. If the fragmented scalar multiplications $\phi_i : F_iR \times M_i \to M$ are induced from an $R$-module $\Omega$, that is, when $M \subset \Omega$, we call $M$ a glider representation. In this case we have that $M_i^* = \{ m \in M,~F_iRm \subset M\}$. If for all $i$ we moreover have that $M_i^* = M_i$, we say that $M$ is natural.\\

Before we introduce the Verma gliders, we mention a few facts about glider representations and fragments over finite algebra filtrations $FR$, see \cite{CVo}. We do not intend this to be too elaborate however. Like in classical representation theory, we are interested in (weakly) irreducible fragments, which are fragments having no non-trivial (strict) subfragments. It can be shown that such fragments $M \supset M_1 \supset \ldots$ have finite essential length $e$, i.e. $M_e \neq B(M)$ with $e$ maximal as such and where $B(M) = \cap_i M_i$ denotes the body of the fragment. In fact, one may assume the body to be zero. This follows from the fact that modding out a strict subfragment preserves irreducibility. Also, the $F_0R$-module $M_e$ determines the fragment completely, since $M_i = F_{e-i}RM_e$ for $0 \leq i \leq e$. From this fact one deduces that if $n$ is the length of the finite algebra filtration, i.e. $F_nR = R$ with $n$ minimal as such, we may restrict to irreducible fragments having essential length $e \leq n$.\\

Consider now a chain $\mf{g}_1 \subset \ldots \subset \mf{g}_n$ of semisimple Lie algebras and an associated chain of maximal toral subalgebras such that all inclusions $\mf{g}_i \subset \mf{g}_{i+1}, (i=1,n-1)$ satisfy condition \eqref{condition}. This fixed chain of Lie algebras determines a positive algebra filtration on the universal enveloping algebra $U(\mf{g}_n)$ given by $F_iU(\mf{g}_n) = U(\mf{g}_{i+1})$ for $i =0, \ldots, n-2$, $F_mU(\mf{g}_n) = U(\mf{g}_n)$ for $m \geq n-1$.  By the remark at the end of the previous section, we can pick bases $\Delta(i)$ of $\mf{g}_i$ such that $\iota(U(\mf{n}_i)) \subset U(\mf{n}_{i+1})$ for $i=1,\ldots,n-1$. In fact, $\iota(\mf{n}_1) = \mf{n}_1^* = \bigoplus_{\alpha^* \in \Phi_1^* \cap \Phi_2^+} g_{2,\alpha^*}$, but we will just write $\mf{n}_1$ instead of $\mf{n}_1^*$. Also, we denote the intersection $\Phi_1^* \cap \Phi_2^+$ by $(\Phi_1^*)^+$.
\begin{definition}
Let $\mf{g}_1 \subset \ldots \subset \mf{g}_n$ be a chain of semisimple Lie algebras as above. A glider representation, of essential length $n-1$, $\Omega \supset M \supset M_1 \supset \ldots \supset M_{n-1} \supset 0 \supset \ldots$ with regard to the finite algebra filtration of length $n-1$ $F_iU(\mf{g}_n) = U(\mf{g}_{i+1})$ on $U(\mf{g}_n)$ is called a Verma glider if it satisfies the following conditions
\begin{enumerate}
\item $\Omega = M(\lambda_n), \lambda_n \in \mf{h}_n^*$ is a $\mf{g}_n$-Verma module,
\item  $M_{n-1} = M(\lambda_1), \lambda_1 \in \mf{h}_1^*$ is a $\mf{g}_1$-Verma module,
\item all $M_{n-1-i}, i=1,\ldots, n-2$ are contained in some $\mf{g}_{i+1}$-Verma module $M(\lambda_{i+1}), \lambda_{i+1} \in \mf{h}_{i+1}^*$.
\end{enumerate}
\end{definition}

One of the purposes of glider theory is to provide information between representations of the various $U(\mf{g}_i)$ appearing in the chain. There is a nice way to construct Verma gliders by starting from a set of Verma modules:
$$\begin{array}{ll}
M(\lambda_1) &\lambda_1: \mf{h}_1 \to K,\\
M(\lambda_2) &\lambda_2: \mf{h}_2 \to K,\\
\quad\vdots & \qquad \vdots\\
M(\lambda_n) &\lambda_n: \mf{h}_n \to K.
\end{array}$$

The Verma module $M(\lambda_n)$ will play the role of $\Omega$, whereas the $M_{n-i}$ will be embedded in $M(\lambda_i)$ for $i=1,\ldots, n$. To establish this we must embed a `smaller' $\mf{g}_1$-Verma module in a `bigger'  $\mf{g}_n$-Verma module. Geometrically, one can think of a Verma module in terms of its weight space $\lambda - \Gamma$, hence such an embedding comes down to viewing $\lambda_1 - \Gamma_1$ as a subset of $\lambda_n - \Gamma_n$, in some sense at least. In fact, we want to do this step by step, that is, we would like to embed $M(\lambda_i)$ inside $M(\lambda_{i+1})$ as a $U(\mf{g}_i)$-module for $i=1,\ldots, n-1$. If we succeed in doing so, $\Omega = M = M(\lambda_n) \supset M(\lambda_{n-1}) \supset \ldots \supset M(\lambda_1)$ becomes a Verma glider. Admittedly, $\Omega = M$ in this particular example, but this need not always be the case, which we will see below. Nonetheless, in section 4, we deal with irreducible gliders, for which it holds that $M = U(\mf{g}_n)M_{n-1}$ is a $U(\mf{g}_n)$ module. For example, when $\lambda_n$ is antidominant, $\Omega = M(\lambda_n)$ is simple, thus equals $M$ if the glider is irreducible. But we are running ahead of things here.\\

Let $\lambda_1 \in \mf{h}_1^*, \lambda_2 \in \mf{h}_2^*$. Since $M(\lambda_1)$ is a highest weight module with highest weight vector, say $v^+_{\lambda_1}$, an embedding $M(\lambda_1) \subset M(\lambda_2)$ as $U(\mf{g}_1)$-modules is given by an element $z \in U(\mf{n}^-_2)$, that is, it is given by sending $v^+_{\lambda_1}$ to $zv^+$, with $v^+$ a highest weight vector of $M(\lambda_2)$. Order the positive roots $\alpha_1,\ldots, \alpha_m$ of $\mf{g}_2$. If $z = \sum_i^{'} \mu_iy_1^{r^i_1}\ldots y_m^{r^i_m} \in U(\mf{n}^-_2)$ (notations from the previous section), we denote by $\wt{z}$ the functional
$$\wt{z}: \mf{h}_2 \to K,~h  \mapsto -\sum_i^{'}\sum_{j=1}^m \mu_i r^i_j\alpha_j(h).$$ 
The element $z$ must satisfy two conditions, namely
\begin{eqnarray}
&\lambda_1 = \pi(\lambda_2 - \wt{z}), \label{cond1} \\
&U(\mf{n}_1)\cdot zv^+ = 0. \label{cond2}
\end{eqnarray}
The first condition depends on the choice of the $\lambda_i ~ (i=1,2)$, but the second one solely depends on the structure of the root systems and, more importantly, on how both $\Phi_1$ and $\Phi_2$ are related! In fact, the second condition will determine for which pairs $(\lambda_1,\lambda_2)$ we obtain Verma gliders. How can we determine the elements $z$ satisfying condition \eqref{cond2}? A starting point would be to determine which of the elements $y_i = y_{\alpha_i}$ for positive roots $\alpha_i$ satisfy. The crucial point will be the following
\begin{lemma}\lemlabel{substract}\cite[Lemma 10.2.A]{Hu}\\
If $\alpha$ is a positive root but not simple, then $\alpha - \beta$ is a (necessarily positive) root for some $\beta \in \Delta.$
\end{lemma}

Let $\Delta(1)^*$ denote the set of roots extending the simple roots of $\Delta(1)$ and then pick a suitable base $\Delta(2)$ such that $\Delta(1)^* \subset \Phi_2^+$. If $\alpha \in \Phi_2^+$ is not simple, then by the lemma there exists a simple root $\beta \in \Delta(2)$ such that $\alpha - \beta$ is a root, hence $[x_\beta, y_\alpha] \neq 0$ in $U(\mf{n}_2^-)$. If $\beta \in \Delta(1)^*$, then $z = y_\alpha$ does not satisfy the second condition. In general, if some root $\gamma \in \Phi_1^*$ is such that $\alpha - \gamma$ is a (positive) root, then $y_\alpha$ does not apply. We will see below that often it holds that $\Delta(1)^* \subset \Delta(2)$, and in this case the previous reasoning immediately shows that $x_\gamma$ with $\gamma \in \Delta(2) \setminus \Delta(1)^*$ an additional simple root satisfies condition \eqref{cond2}. Observe that we made a choice of base $\Delta(1)$, from which we determined a suitable base $\Delta(2)$. Fortunately we have the following
\begin{proposition}\proplabel{amount}
The amount of roots $\gamma \in \Phi_2$ such that $z = y_{\gamma}$ satisfies condition \eqref{cond2} above, is independent of our choice of bases $\Delta(1), \Delta(2)$, i.e. such that $\Delta(1)^* \subset \Phi_2^+$.
\end{proposition}
\begin{proof}
A base change from $\Delta(1)$ to $\Delta(1)'$ is given by an element $\sigma$ of the Weyl group $W_1$ of $\mf{g}_1$. The Weyl group is generated by the $s_\alpha, \alpha \in \Phi_1$, where $s_\alpha$ is the reflection in $\mf{h}_1^*$ with regard to the hyperplane orthogonal to $\alpha$. Algebraically, $s_\alpha(\beta) = \beta - \beta(h_\alpha)\alpha$, from which it follows that $\pi(s_{\alpha^*}(\beta^*)) = s_\alpha(\beta)$. In other words, the element $s_{\alpha^*}$ in the Weyl group $W_2$ of $\mf{g}_2$ maps $\Phi_1^*$ to itself. This shows that to $\sigma$ there is an associated element $\sigma^* \in W_2$, which maps $\Phi_1^*$ to itself. It follows that $\sigma^*(\Delta(2))$ is a suitable base and the result follows easily because $\sigma^*(\beta - \gamma) = \sigma^*(\beta) - \sigma^*(\gamma)$ for $\beta,\gamma \in \Phi_2$. It remains to show that if, given any $\Delta(1)$, two bases $\Delta(2)$, $\Delta(2)'$ suffice, that the amount of roots is still the same. Let $\tau \in W_2$ be the (unique) element sending $\Delta(2)$ to $\Delta(2)'$. In other words, if $\Delta(2) = \Delta(\gamma)$ for a regular $\gamma \in E_2$, then $\Delta(2)' = \Delta(\tau(\gamma))$. Since $\pi(\gamma)$ and $\pi(\tau(\gamma))$ must lie on the same side of each hyperplane $P_\alpha, \alpha \in \Phi_1$ in $E_1$, $\tau$ must be a composition of reflections with regard to a hyperplane containing the space generated by $\Phi_1^*$. This shows that $\tau$ sends $\Phi_1^*$ to itself. Hence if $\beta \in \Phi_2^+$ is such that $\beta - \alpha^*$ is a positive root for $\alpha \in \Phi_1^+$, then $\tau(\beta) - \tau(\alpha^*) = \tau(\beta - \alpha^*)$ is a positive root in $\Phi_2$ with regard to $\Delta(2)'$, and we are done.
\end{proof}

\begin{definition}
An element $z = y_\gamma$ for some positive root $\gamma \in \Phi_2$ satisfying condition \eqref{cond2} is called an embedding element.
\end{definition}

\begin{example}\exalabel{sln}
Consider the embedding $\liesl_n \subset \liesl_m$ ($n < m$) corresponding to the first $n-1$ nodes of the $A_{m-1}$-diagram. Then $\Delta(1)^* = \{L_1-L_2, \ldots, L_{n-1} - L_n\}$. As observed above, the $m-n$ additional basis elements are possible candidates, and for obvious reasons, so are the roots which are positive linear combinations of these. Of course, they form an $A_{m-n}$-diagram and so we already obtain $\frac{(m-n)^2+(m-n)}{2}$ candidates. The $\frac{n^2+n}{2}$ roots from $\liesl_n$ obviously do not apply. The only roots that are left are the ones that go through the last node of $A_{n-1}$ and the first node of $A_{m-n}$. There are $(m-n)n$ roots of this form, that is, with $i \leq n < j$. These roots also don't apply. In total we checked all of them, since
$$\frac{n(n+1)}{2} + \frac{(m-n)(m-n+1)}{2} + n(m-n) = \frac{m(m+1)}{2}.$$ 
So we arrive at the amount of $\frac{(m-n)^2+(m-n)}{2}$.
\end{example}

Of course, if $\gamma \in \Phi_2$ is such that $y_\gamma \in C_{\mf{n}_1}(\mf{n}_2^-) = \{ z \in \mf{n}_2^-~\vline~ [\mf{n}_1,z] = 0\}$, the centralizer of $\mf{n}_1$ inside $\mf{n}_2^-$, then $y_\gamma$ automatically satisfies \eqref{cond2}. However, for some inclusions $\mf{g}_1 \subset \mf{g}_2$ it could be that $\alpha^* = \beta + \gamma$ with $\beta, \gamma \in \Delta(2) \setminus \Phi_1^*$. Since $\beta, \gamma$ are additional roots, $y_\beta, y_\gamma$ are both embedding elements, but they are not in the centralizer $C_{\mf{n}_1}(\mf{n}_2^-)$. Indeed,
$$[x_{\alpha^*}, y_\beta]v^+ = \lambda x_\gamma v^+ = 0, \quad \lambda \in K.$$
Also, we do not have that $y_\beta y_\gamma$ also satisfies \eqref{cond2}. Indeed
\begin{eqnarray*}
[x_{\alpha^*},y_\beta y_\gamma] v^+&=& [x_{\alpha^*},y_\beta]v^+ + y_\beta [x_{\alpha^*},y_\gamma]v^+\\
&=& \lambda y_\gamma v^+ \neq 0.
\end{eqnarray*}
We intend to apply the glider theory first to chains of simple Lie algebras of the same type A,B,C,D with inclusions as in \exaref{Dynkin}. For these inclusions we can choose $\Delta(1), \Delta(2)$ such that $\Delta(1)^* \subset \Delta(2)$, hence we do not have the peculiar behavior from above. Therefore we assume from now on that the inclusion $\mf{g}_1 \subset \mf{g}_2$ satisfies condition \eqref{condition} but that the situation above does not occur. We deduced that for these inclusions we have that $y_\gamma$ is an embedding element if and only if $y_\gamma \in C_{\mf{n}_1}(\mf{n}_2^-)$. 
\begin{lemma}
Let $\beta, \gamma \in \Phi_2^+$ be such that $\beta + \gamma$ is also a root. If two out of three elements of the set $\{y_\beta, y_\gamma, y_{\beta + \gamma}\}$ are embedding elements, then so is the third one.
\end{lemma}
\begin{proof}
Follows directly from the Jacobi identity as $y_{\beta+\gamma} = \lambda [y_\beta,y_\gamma]$ for some $\lambda \in K$.
\end{proof}
In particular, if $\beta, \gamma$ are roots such that their associated elements $y_\beta, y_\gamma$ are embedding elements, then if $\beta+\gamma$ is a root, then $y_{\beta+\gamma}$ is also an embedding element.
\begin{lemma}\lemlabel{root1}
Let $\beta$ and $\gamma$ be positive roots such that $y_\beta$ and $y_\gamma$ are embedding elements, then for all $n,m \geq 0$, $y_\beta^ny_\gamma^m$ also satisfies condition \eqref{cond2}.
\end{lemma}
\begin{proof}
Straightforward, as $y_\beta^ny_\gamma^m \in C_{\mf{n}_1}(U(\mf{n}_2^-)) $.
\end{proof}
\begin{lemma}\lemlabel{root2}
Let $\beta \in \Phi_2$ be a positive root. The following are equivalent
\begin{enumerate}
\item $y_\beta$ is an embedding element,
\item $\forall n >0,~ y_\beta^n \in C_{\mf{n}_1}(U(\mf{n}_2^-))$,
\item $\exists n >0,~ y_\beta^n \in C_{\mf{n}_1}(U(\mf{n}_2^-))$.
\end{enumerate}
\end{lemma}
\begin{proof}
$(1) \Rightarrow (2)$ is \lemref{root1} and $(2) \Rightarrow (3)$ is trivial. Assume that $y_\beta^n \in C_{\mf{n}_1}(U(\mf{n}_2^-))$ for some $n > 1$. From \cite[Lemma 21.4]{Hu} we know that for any $\alpha^* \in (\Phi_1^*)^+$ we have the equality
$$0 = [x_{\alpha^*},y_\beta^n] = n[x_{\alpha^*},y_\beta]y_\beta^{n-1} + {{n}\choose{2}} [[x_{\alpha^*},y_\beta],y_\beta]y_\beta^{n-2} + {{n}\choose{3}}[[[x_{\alpha^*},y_\beta]y_\beta]y_\beta]y_\beta^{n-3}.$$
 If $[x_{\alpha^*},y_\beta] = \lambda y_{\beta - \alpha^*}$ for some $\lambda \in K$, then we would have a linear dependence relation of PBW-polynomials. Indeed, either $ [[x_{\alpha^*},y_\beta],y_\beta] =0$ or equals $y_\gamma$ for some root $\gamma \in \Phi_2^+$. The same holds for$[[[x_{\alpha^*},y_\beta]y_\beta]y_\beta]$, but with a different root $\gamma'$. This is, however, impossible. Observe moreover that we wrote $y_{\beta - \alpha^*}$, because $\alpha^* - \beta$ can not be positive by the assumption on the inclusion $\mf{g}_1 \subset \mf{g}_2$.
\end{proof}
\begin{lemma}\lemlabel{root3}
Let $\gamma, \beta$ be positive roots. Then $y_\beta y_\gamma \in C_{\mf{n}_1}(U(\mf{n}_2^-))$ if and only if $y_\beta, y_\gamma$ are embedding elements.
\end{lemma}
\begin{proof}
We only need to show the `if' direction, so suppose that $y_\beta \notin C_{\mf{n}_1}(\mf{n}_2^-)$ and let $\alpha^* \in (\Phi_1^*)^+$ be such that $0 \neq [x_{\alpha^*}, y_\beta] \in \mf{g}_{2,\alpha^*-\beta}$. By the assumption on the inclusion $\mf{g}_1 \subset \mf{g}_2$, $\alpha^* - \beta$ must be a negative root, whence we can write $[x_{\alpha^*},y_\beta] = \lambda y_{\beta - \alpha^*}$. We obtain
\begin{eqnarray*}
0 = [x_{\alpha^*},y_\beta y_\gamma] &=& x_{\alpha^*}y_\beta y_\gamma - y_\beta y_\gamma x_{\alpha^*}\\
&=& y_\beta x_{\alpha^*}y_\gamma + \lambda y_{\beta - \alpha^*} y_\gamma - y_\beta y_\gamma x_{\alpha^*}\\
&=& y_\beta [x_{\alpha^*},y_\gamma] + \lambda y_{\beta - \alpha^*} y_\gamma.
\end{eqnarray*}
It follows that $[x_{\alpha^*},y_\gamma] \neq 0$, which means that it equals $\mu y_{\gamma - \alpha^*}$ for some $\mu \in K$. Hence we obtain the equality 
$$\lambda y_{\beta - \alpha^*}y_\gamma = -\mu y_\beta y_{\gamma - \alpha^*}$$
in $U(\mf{n}_2^-)$. If the monomials on both the left and right hand side are PBW-monomials, then $\beta - \alpha^* = \beta$, which is absurd. If, say, the right hand side is not written in PBW form, then 
$$\lambda y_{\beta - \alpha^*}y_\gamma = -\mu y_{\gamma - \alpha^*}y_\beta + \rho[y_\beta, y_{\gamma - \alpha^*}].$$
If the bracket $[y_\beta, y_{\gamma - \alpha^*}] =0$, then $\beta = \gamma$. But then it follows from \lemref{root2} that $y_\beta$ is an embedding element, contradicting our assumption. If the bracket is not zero, we obtain a linear dependence relation between three PBW-monomials which is impossible. Hence, our contradiction is wrong, i.e. $y_\beta$ is an embedding element. It follows then that the same is true for $y_\gamma$.
\end{proof}
\begin{proposition}\proplabel{PBW}
Let $z = y_{\alpha_1}^{r_1}\ldots y_{\alpha_m}^{r_m}$ be a PBW-monomial in $U(\mf{n}_2^-)$. Then $z \in C_{\mf{n}_1}(U(\mf{n}_2^-))$ if and only if all $y_{\alpha_i}$ are embedding elements $(i=1,\ldots, m)$.
\end{proposition}
\begin{proof}
We proceed by induction on the number $n$ of positive roots $\alpha_i$ appearing in the PBW-monomial $z$. The case $n=1$ is just a restatement of \lemref{root2}. Assume now that the result holds for $n-1$ and let $z = y_\alpha^{m}z'$, where $z'$ has $n-1$ roots $\alpha_i \neq \alpha$ appearing. Suppose that $y_\alpha$ is not an embedding element and let $\beta^* \in (\Phi_1^*)^+$ be such that $[x_{\beta^*},y_\alpha^m] =\sum_k^{'} Y_k$ is a non-zero sum of PBW-monomials. We have the equality
$$0 = [x_{\beta^*},z] = y_\alpha^m [x_{\beta^*},z'] + \sum_{k}^{'} Y_k z'.$$
As in the proof of \lemref{root3}, we conclude that $[x_{\beta^*},z'] = \sum_l^{'} Z_l$ is a non-zero sum of PBW-monomials and we obtain the equality
$$- y_\alpha^m \sum_l^{'} Z_l = \sum_k^{'} Y_k z'.$$
By the proof of \lemref{root2} we know that no $y_\alpha^m$ appears in the monomials $Y_k$ and it also does not appear in $z'$, so we arrive at a contradiction. We conclude that $y_\alpha$ is an embedding element, whence $z' \in C_{\mf{n}_1}(U(\mf{n}_2^-))$. The result now follows via induction.
\end{proof}
The previous results hint at a possibility for the centralizer $C_{\mf{n}_1}(U(\mf{n}_2^-))$ to be generated by the embedding elements $y_\gamma$. 
The only thing we still need to check is that when the sum $z_1 + z_2$ of two PBW-monomials is in $C_{\mf{n}_1}(U(\mf{n}_2^-))$, then so are both elements $z_1, z_2$. If the weight $z_1$ is different from the weight of $z_2$, this is trivial. Also, even if $z_1, z_2$ have the same weight, but different degree, then it follows again that both $z_1, z_2$ are in the centralizer. By the degree of a PBW monomial $z =  y_1^{r_1}\ldots y_m^{r_m}h_1^{s_1}\ldots h_n^{s_n}x_1^{t_1}\ldots x_m^{t_m}$ we mean the sum $\sum_{i=1}^m (t_i - r_i)$. To see this, suppose that $z_1 =  y_{\alpha_1}^{r_1}\ldots y_{\alpha_m}^{r_m}, z_2 =  y_{\alpha_1}^{s_1}\ldots y_{\alpha_m}^{s_m}$, with $- \sum_{i=1}^m r_i <- \sum_{i=1}^m s_i$ and such that $z_1 + z_2 \in C_{\mf{n}_1}(U(\mf{n}_2^-))$ but both $z_1,z_2 \in C_{\mf{n}_1}(U(\mf{n}_2^-))$. Then there exists an $\alpha^* \in (\Phi_1^*)^+$ such that
$$ 0 \neq  [x_{\alpha^*},z_1] = -[x_{\alpha^*}, z_2].$$
When expressing the left hand side of the above equality as a linear combination of PBW-monomials there appears exactly one monomial of degree $1- \sum_{i=1}^m r_i$, namely $z_1x_{\alpha^*}$ and all the other PBW-monomials in this combination have higher degree. Similarly, in the expression of the right hand side, the lowest degree appearing is $1 - \sum_{i=1}^m s_i > 1- \sum_{i=1}^m r_i$, contradiction. The only ingredient missing for proving that the embedding elements are a generating set, is that when $z_1,z_2$ have the same degree. Unfortunately, we have a counterexample.
\begin{example}\exalabel{fail}
Consider the inclusion $\liesl_2 \subset \liesl_4$ embedded in the top left hand corner. We know from \exaref{sln} that $y_{L_1-L_4}$ and $y_{L_1-L_3}$ are not embedding elements. Hence \lemref{root3} entails that $z_1 = E_{14}\ot E_{23} = y_{L_1-L_4}y_{L_2-L_3}$ and $z_2 = E_{24}\ot E_{13} = y_{L_2 - L_4}y_{L_1-L_3}$ are not in the centralizer. For $\alpha^* = L_1 - L_2$ we have
$$[E_{12}, E_{41}\ot E_{32} - E_{42}\ot E_{31}] = -E_{42}\ot E_{32} + E_{42}\ot E_{32} = 0,$$
which shows that $z_1 - z_2 \in C_{\mf{n}_1}(U(\mf{n}_2^-))$.
\end{example}

Now, let us finally go back to the construction of Verma gliders. For chains of length two, that is, just an inclusion $\mf{g}_1 \subset \mf{g}_2$, we must embed a $\mf{g}_1$-Verma module $M(\lambda_1)$ in a $\mf{g}_2$-Verma module $M(\lambda_2)$ by means of an element $z \in U(\mf{n}_2^-)$ that satisfies \eqref{cond2}. For chains of bigger length we add a small remark.

\begin{remark}\remlabel{remark}
Our rather long digression on the centralizer $C_{\mf{n}_1}(U(\mf{n}_2^-))$ showed that `many' elements $z \in U(\mf{n}_2^-)$ satisfy \eqref{cond2}. Also, if $z \in U(\mf{n}_2^-)$ satisfies, then so does $kz$ for any $k \in K$. However, there is a good reason why we always want $k = 1$. Indeed, if we look at chains of Lie algebras of bigger length, say 3, then we consider functionals 
 $$\lambda_i: \mf{h}_i \to K, \quad i=1,2,3.$$
The idea remains the same: we want to embed $M(\lambda_1)$ inside $M(\lambda_2)$ as $U(\mf{g}_1)$-modules and $M(\lambda_2)$ inside $M(\lambda_3)$ as $U(\mf{g}_2)$-modules, such that the composition embeds $M(\lambda_1)$ as an $U(\mf{g}_1)$-module as well. Suppose that $s,r \in K$, $y\in U(\mf{n}_2^-)$ and $z \in U(\mf{n}_3^-)$ are such that 
$$v_{\lambda_1}^+ \mapsto sy v_{\lambda_2}^+, \quad v_{\lambda_2}^+ \mapsto rz v_{\lambda_3}^+$$
give the right embeddings. Composition is given by $v_{\lambda_1}^+ \mapsto rsyz v_{\lambda_3}^+$ but for $h_1 \in \mf{h}_1$ we have that
\begin{eqnarray*}
\lambda_1(h) &=& \lambda_2(h) - s\wt{y}(h)\\
&=& \lambda_3(h) - r\wt{z}(h) - s\wt{y}(h)\\
&=& \lambda_3(h) -rs(\wt{y}(h) + \wt{z}(h)).
\end{eqnarray*}
Hence we must have that $r = s =rs$ or that $r= s= 1$.
\end{remark}
\begin{example}\exalabel{sl234}
Consider the tower $\mf{sl}_2 \subset \mf{sl}_3 \subset \mf{sl}_4$, with embeddings graphically depicted as
$$\left(\begin{array}{cc|c|c}
\ast_1 & \ast_1 & \ast_2 & 0\\
\ast_1 & \ast_1 & \ast_2 & 0\\
\hline
\ast_2 & \ast_2 & \ast_2 & 0\\
\hline
0 & 0 & 0 &0
\end{array}\right)$$
 With notations as before, let $\Delta(1)^* = \{L_1-L_2\}, \Delta(2)^* = \{L_1 - L_2, L_2 - L_3\} $ and $\Delta(3) = \{L_1-L_2,L_2-L_3,L_3-L_4\}$ and the dual basis of $\{L_i\}$ for $\mf{h}$ is denoted by $\{H_1,\ldots, H_4\}$. For the first embedding, the only candidate is $\alpha = L_2-L_3$ and for the second we have $\beta = L_3 - L_4$. Hence $(\lambda_2 - \lambda_1)(H_1-H_2) = \alpha(H_1 - H_2) = -1$ and 
 $(\lambda_3 - \lambda_2)(H_1 - H_2,~H_2 - H_3) = \beta(H_1-H_2,~H_2-H_3) = (0~-1)$. For example, take $\lambda_1 = -1$, $\lambda_2 = (2~1)$ and 
 $\lambda_3 = (2~ 0 ~0)$
  where the functionals are represented with regard to the basis $\{H_1 - H_2, H_2 - H_3, H_3-H_4\}$. 
So we have $M_2 = U(\mf{sl}_2)y_{L_2-L_3}y_{L_3-L_4}v^+$. The embedding of $M(\lambda_2)$ into $\Omega$ uses the simple root $L_3 - L_4$ and for simple roots $\alpha$ we can easily check whether $s_\alpha \cdot \lambda_3 < \lambda_3$, i.e. $\lambda_3 - s_\alpha \cdot \lambda_3 \in \Gamma_3$. Recall that for an element $\sigma$ of the Weyl group and $\lambda \in \mf{h}^*$, the dot action $\sigma \cdot \lambda = \sigma(\lambda + \rho) - \rho$ with $\rho$ half the sum of the positive roots. Indeed, it suffices that 
$$<\lambda_3, (L_3-L_4)^\vee> = \lambda_3(H_3 - H_4) = 0 \in \mathbb{Z}^+.$$
There are two conditions on $M_1$, namely
$$M_2 = U(\mf{sl}_2)y_{L_2-L_3}y_{L_3-L_4}v^+ \subset U(\mf{sl}_3)y_{L_2-L_3}y_{L_3-L_4}v^+ \subseteq M_1,$$
and
$$ M_1 \subseteq U(\mf{sl}_3)y_{L_3-L_4}v^+ \subset M(s_{L_3-L_4}\cdot \lambda_3) \subset  M \subset \Omega = M(\lambda_3).$$ 
The second condition comes from the fact that we want $M_1$ to be contained in the $\liesl_3$-Verma module $M(\lambda_2)$. For $M_1$ we can choose any $U(\mf{sl}_2)$-module satisfying both conditions. For example, $M_1 = M(s_{L_3-L_4}\cdot \lambda_3)$ satisfies.
For $M$ we can add the $y_{L_1-L_2}$-string starting from $y_\alpha v^+$ with $\alpha$ any positive root in $\Phi_3 \setminus \Phi_2^*$. For example
\begin{eqnarray*}
\Omega = M(\lambda_3) &\supset& M =  M(s_{L_3-L_4}\cdot \lambda_3) + U(\sl_2)y_{L_1 - L_4}v^+ \\
&\supset& M(s_{L_3-L_4}\cdot \lambda_3) \supset U(\mf{sl}_2)y_{L_2-L_3}y_{L_3-L_4}v^+.
\end{eqnarray*}
Another Verma glider would be
\begin{eqnarray*}
\Omega = M(\lambda_3) &\supset& M(s_{L_3-L_4}\cdot \lambda_3) + U(\sl_2)y_{L_1 - L_4}v^+ \\
&\supset& U(\mf{sl_3})y_{L_3 - L_4}v^+ \supset U(\mf{sl}_2)y_{L_2-L_3}y_{L_3-L_4}v^+.
\end{eqnarray*}
\end{example}

\begin{example}\exalabel{trivial}
In the special case that $(\lambda_i)_{\vline \mf{h}_{i-1}} = \lambda_{i-1}$ for $i = 2, \ldots ,n$, we can take the element $z = 1$  to define the embedding at every stage. We obtain the glider representation
$$M(\lambda_n ) = U(\mf{u}_n^-) \ot \mathbb{C}_{\lambda_n} \supset  U(\mf{u}_{n-1}^-) \ot \mathbb{C}_{\lambda_n}  \supset \cdots \supset U(\mf{u}_1^-) \ot \mathbb{C}_{\lambda_n}.$$  
\end{example}

\section{Irreducible gliders}

The notion of irreducibility for glider representations is introduced in \cite{EVO} and extended in \cite{CVo}. The content concerning irreducible gliders is already reviewed briefly after \defref{fragment}. To investigate when Verma gliders are irreducible, we recall the notion of antidominant weights. A functional $\lambda \in \mf{h}^*$ is called antidominant if $<\lambda + \rho, \alpha^\vee> \not\in \mathbb{Z}^{>0}$, where $\rho$ is half the sum of the positive roots (or the sum of the fundamental weights $\ov{\omega}_i$, which are obtained by base change via the Cartan matrix). Antidominant weights play an important role in the study of Verma modules since $M(\lambda)$ is simple if and only if $\lambda$ is antidominant. We also recall the following important theorem due to Verma

\begin{theorem}\thelabel{Verma}
Let $\lambda \in \mf{h}^*$. Suppose that $\mu := s_\alpha \cdot \lambda \leq \lambda$ for some $\alpha \in \Phi^+$. Then there exists an embedding $M(\mu) \subset M(\lambda)$.
\end{theorem}
Moreover, since $\dim\Hom(M(\mu),M(\lambda)) \leq 1$ for all $\mu,\lambda$ the above embedding is unique up to some scalar.
In the particular situation of \exaref{trivial}, the answer whether a Verma glider is irreducible follows from the classical representation theory of Verma modules.
\begin{proposition}\proplabel{irre}
With assumptions and notations of \exaref{trivial}, we have that $M(\lambda_n)$, as a fragment, is an irreducible fragment if and only if $ \lambda_1 \in \mf{h}_1^*$ is antidominant.
\end{proposition}
\begin{proof}
If $M(\lambda_n)$ is irreducible, then $U(\mf{u}_1^-) \ot \mathbb{C}_{\lambda_n}$ must be a simple $U(\mf{g}_1)$-module. But this module is just the ordinary Verma module $M(\lambda_1)$ and so $\lambda_1$ is antidominant. The converse follows easily since by definition $U(\mf{g}_i)M^{\mf{g}_1}(\lambda_1) = U(\mf{u}_{i}^-) \ot \mathbb{C}_{\lambda_n}$, for all $i = 1,\ldots,n$ (the upper index in $M^{\mf{g}_1}(\lambda_1)$ means that we consider the $\mf{g}_1$-Verma module).
\end{proof}

\begin{example}\exalabel{sl}
Consider $\mf{sl}_2 \subset \mf{sl}_3$ with the embedding as in \exaref{sl234}. The root vectors $L_3-L_1 = -\rho$ and $L_2-L_3$ both restrict to $\frac{-1}{2}(L_1-L_2)$ on $\mf{h}_1$, which is antidominant (every antidominant weight $\lambda$ is minimal in its linkage class $W \cdot \lambda$ and in $\mf{sl}_2$ only $\lambda$ and $-\lambda - 2$ are linked). By the proposition, both Verma gliders
$$ \begin{array}{c}
\Omega = M = M( - \rho) \supset  M^{\mf{sl}_2}( -\frac{1}{2}(L_1-L_2)),\\
\Omega = M = M(L_2-L_3) \supset  M^{\mf{sl}_2}( -\frac{1}{2}(L_1-L_2))
\end{array}$$
 are irreducible. However, since $<L_2-L_3+ \rho , (L_2-L_3)^\vee> = 3$, $L_2-L_3$ is not antidominant. Of course, $- \rho$ is antidominant. 
\end{example}

However, when there appear non-trivial embedding elements, \propref{irre} is no longer true. Of course, it is a necessary condition for a fragment of essential length $n-1$ to be irreducible is that $M_{n-1}/B(M)$ is a simple $F_0R$-module. By definition, this $M_{n-1} = M(\lambda_1)$, so irreducibility of the glider indeed implies that $\lambda_1$ is antidominant. 

\begin{example}
Consider $\liesl_2 \subset \liesl_3$ embedded in the top left corner. If $\lambda_1 = -\frac{1}{2}(L_1-L_2)$ and $\lambda_2 = 3(L_2- L_3)$, then $z = y_{L_2-L_3}^2$ is an embedding element that satisfies. Hence we have the Verma glider
$$\Omega = M(3(L_2-L_3)) = M \supset U(\liesl_2)y_{L_2 - L_3}^2v_{\lambda_1}^+.$$
Since $\langle3(L_2-L_3),(L_2-L_3)^\vee\rangle = 6$, we know that $M(s_{L_2-L_3}\cdot \lambda_2 ) = U(\liesl_3)y_{L_2-L_3}^7 \subset M(\lambda_2)$. Hence $M(s_{L_2-L_3}\cdot \lambda_2 )  \supset U(\liesl_2)y_{L_2 - L_3}^2v_{\lambda_1}^+$ is a non-trivial subfragment. 
\end{example}
This example shows that the study of irreducible gliders, even for chains of length 2, is already of a much higher complexity. In some cases, however, we can say something more. To state the result we recall that each Verma module $M(\lambda)$ has a unique maximal submodule $N(\lambda)$ and unique simple quotient $L(\lambda) = M(\lambda)/N(\lambda)$. It is a natural question to ask for which $\lambda$ the simple quotients are finite dimensional. To this extent, we recall the notion of dominant integral weights.\\

The root system $\Phi$ of $\mf{g}$ determines a root lattice $\Lambda_r$, which is just the $\mathbb{Z}$-span of $\Phi$. There is also a natural dual lattice, called the integral weight lattice $\Lambda$ defined by
$$\Lambda = \{ \lambda \in \mf{h}^*~|~ <\lambda, \alpha^\vee> \in \mathbb{Z} {\rm~for~all~} \alpha \in \Phi\}.$$
Clearly, $\Lambda_r \subset \Lambda$ and their quotient $\Lambda / \Lambda_r$ is a finite group, the fundamental group of the Lie algebra $\mf{g}$. The subset $\Lambda^+$ denotes the set of elements of $\Lambda$ for which the inproduct is nonnegative for all $\alpha \in \Phi$ . We call $\Lambda^+$ the set of dominant integral weights. Their importance is given by

\begin{theorem}\thelabel{findim} \cite[Theorem 1.6]{Hu2}\\
The simple module $L(\lambda)$ is finite dimensional if and only if $\lambda \in \Lambda^+$.
\end{theorem}

Now, from \propref{amountABCD} we know that if $\mf{g}_1 \subset \mf{g}_2$ is an embedding of simple Lie algebras of type A,B or D and such that $\rk(\mf{g}_2) = \rk(\mf{g}_1) + 1$, that there is only one embedding element $y_\alpha$ with $\alpha$ the additional simple root $\alpha \in \Delta(2) \setminus \Delta(1)^*$. So if $z$ is a PBW-monomial in $U(\mf{n}_2^-)$ that determines the embedding, then $z = y_\alpha^n$ by \propref{PBW}. We have
\begin{proposition}\proplabel{abd}
Let $\lambda_2 \in \Lambda_2^+, \lambda \in \mf{h}_1^*$ and $z = y_\alpha^n$ be such that $\Omega = M(\lambda_2) \supset M \supset U(\mf{g}_1)y_{\alpha}^nv^+$ is a Verma glider. Then the Verma glider is irreducible if and only if $\lambda_1$ is antidominant and $n < m = \langle \lambda_2, \alpha^\vee \rangle$.
\end{proposition}
\begin{proof}
Suppose that $\lambda_1$ is antidominant and $n < m$. If $M \supset M_1$ is not irreducible then $U(\mf{g}_2)y_\alpha^nv^+ \subsetneq M$. In particular, it follows that
$$U(\mf{g}_2)y_\alpha^nv^+ \subset N(\lambda_2) = \sum_{\alpha_i \in \Delta(2)} M(s_{\alpha_i}\cdot \lambda_2).$$
 Because $n < m$, we have that $y_\alpha^nv^+ \notin M(s_\alpha \cdot \lambda_2)$, hence $y_\alpha^nv^+ \notin N(\lambda_2)$, contradiction. Conversely, if $M \supset M_1$ is irreducible, then $\lambda_1$ must be antidominant. Also, since $U(\mf{g}_2)y_\alpha^mv^+ = M(s_\alpha \cdot \lambda_2) \subsetneq M = U(\mf{g}_2)y_\alpha^nv^+$, we have that $m > n$.
 \end{proof}

For general embeddings $\mf{g}_1 \subset \mf{g}_2$ we have the following.

\begin{proposition}\proplabel{domint}
Let $\mf{g}_1 \subset \mf{g}_2$ be a chain of semisimple Lie algebras and suppose that $\lambda_i \in\mf{h}_i^*,~i=1,2$ are such that a Verma glider $\Omega= M(\lambda_2) \supset M \supset M(\lambda_1)$ exists. If $\lambda_2$ is dominant integral and $\lambda_1$ is not, then $M(\lambda_1) = M_1 \subset N(\lambda_2)$ and $U(\mf{g}_2)M(\lambda_1) \subset M \cap N(\lambda_2)$.
\end{proposition}
\begin{proof}
Since $N(\lambda_2) \cap M_1$ is an $U(\mf{g}_1)$-submodule of $M_1 = M(\lambda_1)$, we have that $N(\lambda_2) \cap M_1 \subseteq N(\lambda_1)$ or that $N(\lambda_2) \cap M_1 = M_1$. Since $\lambda_2$ is dominant integral we have that $L(\lambda_2)$ is finite dimensional. If the first case holds, we have an isomorphism of vector spaces
$$ \frac{M_1}{N(\lambda_1)} \cong \frac{\frac{M_1}{N(\lambda_2) \cap M_1}}{\frac{N(\lambda_1)}{N(\lambda_2) \cap M_1}}.$$
Since $M_1/(M_1\cap N(\lambda_2))$ embeds in $M(\lambda_2)/N(\lambda_2)$ it follows that $\lambda_1$ is dominant integral, a contradiction. So we have that $M_1 \subset N(\lambda_2)$. The last statement then automatically follows by the definition of a glider representation.
\end{proof}
\begin{corollary}\corlabel{domint}
In the situation of the previous proposition, a Verma glider $\Omega = M(\lambda_2) \supset M \supset M(\lambda_1)$ with $\Omega = U(\mf{g}_2)M$ and $\lambda_2$ dominant integral, is never irreducible.
\end{corollary}
\begin{proof}
By the previous proposition and by the hypothesis $\Omega = U(\mf{g}_2)M$, the glider representation  $N(\lambda_2) \cap M \supset M(\lambda_1)$ is a non-trivial subfragment. 
\end{proof}


For chains $\mf{g}_1 \subset \ldots \subset \mf{g}_n$ we restrict to Verma gliders of the form 
$$\Omega = M =  M(\lambda_n) \supset M(\lambda_{n-1}) \supset \ldots \supset M(\lambda_1),$$
for functionals $\lambda_i \in \mf{h}_i^*$. If $z_i \in U(\mf{n}_{i+1}^-)$ is the element that determines the embedding $M(\lambda_i)$ inside $M(\lambda_{i+1})$, then we can rewrite such a glider as
\begin{equation}\label{verma}
 M = U(\mf{g}_n)v^+ \supset U(\mf{g}_{n-1})z_{n-1}v^+ \supset \ldots \supset U(\mf{g}_2)z_2z_3\ldots z_{n-1}v^+ \supset U(\mf{g}_1)z_1\ldots z_{n-1}v^+.
 \end{equation}
Again, if all Lie algebras $\mf{g}_i$ are of the same type A,B or D and $\rk(\mf{g}_{i+1}) = \rk(\mf{g}_i) + 1$ and the $z_i$ are PBW-monomials, then $z_i = y_{\alpha_i}^{k_i}$ with $\alpha_i \in \Delta(i+1) \setminus \Delta(i)^*$ the additional simple root. We have the generalization of \propref{abd}.
\begin{theorem}\thelabel{gene}
Consider a Verma glider of the form \eqref{verma} with $\lambda_i \in \Lambda_i^+$ for $i=2,\ldots, n$. The Verma glider is irreducible if and only if $\lambda_1$ is antidominant and $k_i < m_i =  \langle \lambda_{i+1}, \alpha_i^\vee \rangle$ for all $i = 1, \ldots, n-1$.
\end{theorem}
\begin{proof}
Analogous to the proof of \propref{abd}.
\end{proof}
\section{Nilpotent orbits}

The construction of Verma gliders lead to the existence of embedding elements. Obviously, these embedding elements are nilpotent elements, hence belong to some nilpotent orbit. For a complex semisimple Lie algebra $\mf{g}$ these nilpotent orbits are classified by the Dynkin-Kostant classification, see e.g. \cite[Chapter 3]{CoMc} for a nice overview. For $\mf{g}_1 \subset \mf{g}_2$ two complex semisimple Lie algebras, we ask ourselves which nilpotent orbits we reach by just looking at the embedding elements. We restrict here to Lie algebras of the same type $A, B,C$ and $D$ and embeddings as in \exaref{Dynkin}. In \exaref{sln} we determined the embedding elements for type $A$. One can perform similar reasonings for the other types to obtain

\begin{proposition}\proplabel{amountABCD}
Let $\mf{g}_1 \subset \mf{g}_2$ be a canonical embedding of simple Lie algebras of the same type $A,B,C$ or $D$ of rank $n < m$. The amount of embedding elements for each type is given by
$$ \begin{array}{ll}
{\rm~type~} A: & \frac{(m-n)^2+(m-n)}{2}\\
{\rm~type~} B: & (m-n)^2\\
{\rm~type~} C: & (m-n)^2 + (m-n)\\
{\rm~type~} D: & (m-n)^2
\end{array}$$
\end{proposition}

Let us start by looking at type $A$, i.e. at $\liesl_n \subset \liesl_m$. In this case, the classification is given by partitions of $m$. We introduce some notation (following \cite{CoMc}).\\

A partition of $m$ is a tuple $[d_1^{i_1},d_2^{i_2},\ldots, d_k^{i_k}]$ with $d_j$ and $i_j$ positive integers such that 
$$d_1 \geq d_2 \geq \ldots \geq d_k > 0  {\rm ~and~} i_1d_1 + i_2d_2 + \ldots + i_kd_k = m.$$
For a positive integer $i$, we denote the elementary Jordan block of type $i$ by
$$J_i = \left( \begin{array}{cccccc}
0 & 1 & 0 & \ldots & 0 & 0\\
0 & 0 & 1 & \ldots & 0 & 0\\
\vdots & \vdots & \vdots & \ddots & \vdots & \vdots\\
0 & 0 & 0 & \ldots & 0 & 1\\
0 & 0 & 0 & \ldots & 0 & 0
\end{array}\right)$$
For a partition $[d_1^{i_1},d_2^{i_2},\ldots, d_k^{i_k}]$ of $m$ we form the diagonal sum of elementary Jordan blocks 
$$X_{[d_1^{i_1},d_2^{i_2},\ldots, d_k^{i_k}]} = \left( \begin{array}{ccccc}
J_{d_1} & 0 & 0 & \ldots & 0\\
0 & J_{d_1} & 0 & \ldots & 0\\
\vdots & \vdots & \vdots & \ddots & \vdots\\
0 & 0 &0 & \ldots & J_{d_k}
\end{array}\right)$$
where there are $i_1$ blocks $J_{d_1}$, $i_2$ blocks $J_{d_2}$, etc. 
The matrix $X_{[d_1^{i_1},d_2^{i_2},\ldots, d_k^{i_k}] }\in \liesl_m$ is nilpotent and generates the nilpotent orbit $\mc{O}_{[d_1^{i_1},d_2^{i_2},\ldots, d_k^{i_k}]} = PSL_m \cdot X_{[d_1^{i_1},d_2^{i_2},\ldots, d_k^{i_k}]}$.\\


We denote by $H_{\liesl_n}$ the Hasse diagram of the nilpotent orbits of $\liesl_n$. E.g. $H_{\liesl_3}$ is given by\\

\begin{tikzpicture}[scale=.5]
\node (1) at (0,0) {$[3]$};
\node (2) at (0,-2) {$[2,1]$};
\node (3) at (0,-4) {$[1^3]$};
\draw (1) -- (2) -- (3);
\end{tikzpicture}\\

If $\iota: \mf{g}_1 \hookmapright{} \mf{g}_2$ is an inclusion of complex semisimple Lie algebras then we denote by $\iota(H_{\mf{g}_1})$ the Hasse subdiagram of $H_{\mf{g}_2}$ containing those orbits $\mc{O}$ that have an element $\iota(X)$, $X \in \mf{g}_1$ nilpotent.
\begin{lemma}\lemlabel{Jordan}
Let $A \in M_m(\mathbb{C})$ have Jordan normal form $J(A) = J$, then the Jordan normal form of 
$$ J(\left( \begin{array}{c|c}
A & 0\\
\hline 
0 & 0\end{array}\right))
= \left( \begin{array}{c|c}
J & 0\\
\hline 
0 & 0\end{array}\right).$$
\end{lemma}
\begin{proof}
If $J = S^{-1}AS$, then 
$$\left( \begin{array}{c|c}
J & 0\\
\hline 
0 & 0\end{array}\right) = \left( \begin{array}{c|c}
S^{-1} & 0\\
\hline 
0 & I\end{array}\right)\left( \begin{array}{c|c}
A & 0\\
\hline 
0 & 0\end{array}\right)\left( \begin{array}{c|c}
S & 0\\
\hline 
0 & I\end{array}\right).$$
\end{proof}
\begin{remark}\remlabel{zero}
In the lemma, the 0 can also denote any matrix of size $n \times m$ with all entries zero.
\end{remark}

\begin{proposition}\proplabel{Hasseinclusion}
Let $\liesl_n \subset \liesl_m$ then $\iota(H_{\liesl_n})$ is classified by the partitions $[d_1^{i_1},d_2^{i_2},\ldots, d_k^{i_k}]$ of $m$ with $d_k = 1$ and $i_k \geq m-n$.
\end{proposition}
\begin{proof}
Let $X = X_{[d_1^{i_1},d_2^{i_2},\ldots, d_k^{i_k}]} \in \liesl_n$ be the nilpotent element orbit associated to the partition $[d_1^{i_1},d_2^{i_2},\ldots, d_k^{i_k}]$ of $n$. Under the inclusion, $X$ is sent to $X_{[d_1^{i_1},d_2^{i_2},\ldots, d_k^{i_k},1^{m-n}]}$. Conversely, if $[d_1^{i_1},d_2^{i_2},\ldots, d_{k-1}^{i_{k-1}},1^k]$ is a partition of $m$ with $ k \geq m-n$, then the first upper diagonal of $X_{[d_1^{i_1},d_2^{i_2},\ldots, d_{k-1}^{i_{k-1}},1^k]}$ has zeroes on the last $m-n $-entries, hence it belongs to $\iota(\liesl_n)$.
\end{proof}

By \propref{amountABCD} we know that there are $((m-n)^2 +(m-n))/2$ embedding elements. The embedding elements are situated in the following positions:\\

\begin{equation}\label{star}
\left( \begin{array}{ccc|ccccc}
0 &&&&&&&\\
& \ddots &&&&&&\\
&& 0 &&&&&\\
\hline
&&& 0 & \ast & \ast & \ldots & \ast\\
&&& & 0 & \ast & \ldots & \ast\\
&&& & & \ddots & \ddots & \vdots\\
&&& &&& 0 & \ast\\
&&&&&&& 0 
\end{array}\right)
\end{equation}

If we only make linear combinations of the $m-n$ elements on the first upper diagonal then we obtain the partitions of $m-n+1$. \propref{Hasseinclusion} shows that these form $\iota(H_{\liesl_{m-n+1}})$.

\begin{theorem}
Let $\liesl_n \subset \liesl_m$, then the nilpotent orbits generated by the embedding elements correspond to the inclusion of the Hasse subdiagram $\iota(H_{\liesl_{m-n+1}})$.
\end{theorem}
\begin{proof}
We already observed that we reach $\iota(H_{\liesl_{m-n+1}})$. If $X$ is a nilpotent element with only non-zero coefficients on the $\ast$-positions, then \lemref{Jordan} shows that $X$ has Jordan normal form with (possibly) only 1's on the starred positions in \eqref{star}. This shows that $X \in \mc{O}_{[d_1^{i_1},d_2^{i_2},\ldots, d_k^{i_k}]}$, with $d_k = 1$ and $i_k \geq n - 1$. Thus $\mc{O}_X \in \iota(H_{\liesl_{m-n+1}})$ by \propref{Hasseinclusion}. 
\end{proof}

Next, we consider Lie algebras of type $C$. Recall from \cite[Theorem 5.1.3]{CoMc} that the nilpotent orbits of $\liesp_{2m}$ are in one-to-one correspondence with the set of partitions of $2m$ in which odd parts occur with even multiplicity. The root system of $\liesp_{2m}$ is $\{\pm L_i \pm L_j , \pm 2L_i~|~ 1 \leq i, j \leq m, ~ i \neq j\}$ and we make the standard choice $\{L_i \pm L_j, 2L_k|~ 1 \leq i < j \leq m, 1 \leq k \leq m\}$ of positive roots. With respect to this choice of basis, the authors give in \cite{CoMc} a recipe for constructing a standard triple $\{X,H,Y\}$ associated to a partition ${\bf d}$ of $2m$. We quickly recall this. Given ${\bf d} \in \mc{P}(2m)$, break it up into chunks of the following two types: pairs $\{2r+1,2r+1\}$ of equal odd parts, and single even parts $\{2q\}$. We attach sets of positive (but not necessarily simple) roots to each chunk $\mc{C}$ as follows. If $\mc{C} = \{2q\}$, choose a block $\{j+1,\ldots, j+q\}$ of consecutive indices and let $\mc{C}^+ = \mc{C}^+(2q) = \{L_{j+1} - L_{j+2,}L_{j+2}-L_{j+3},\ldots, L_{j+q-1}-L_{j+q,}2L_{j+q}\}$. If $\mc{C} = \{2r+1,2r+1\}$, choose a block $\{l+1,\ldots, l+2r+1\}$ of consecutive indices and let $\mc{C}^+ = \mc{C}^+(2r+1,2r+1) = \{L_{l+1} -L_{l+2},\ldots, L_{l+2r}-L_{l+2r+1}\}$. (Note that $\mc{C}^+$ is empty if $\mc{C} = \{1,1\}$). We define $X$ to be the sum of the $X_\alpha$, where $\alpha$ appears in some of the $\mc{C}^+$.\\

The inclusion $\liesp_{2n} \subset \liesp_{2m}$ sends a matrix $X \in \liesp_{2n}$ to $\wt{X}$ by adding some nonzero rows and columns. Let $X$ be the nilpotent element associated to a partition ${\bf d} = [d_1^{n_1},\ldots, d_k^{n_k}] \in \mc{P}(2n)$ given by the above procedure. Since $\mc{C}^+(1,1) = \emptyset$, a chunk of the form $\{1^{2k}\}$ of ${\bf d}$ does not contribute to $X$. Hence we see that $\wt{X}$ is the associated element of the partition $[d_1^{n_1},\ldots, d_k^{n_k}, 1^{2(m-n)}]$. This gives the analogue of \propref{Hasseinclusion} for type $C$.  By \propref{amountABCD} we know that there are $(m-n)^2 + (m-n)$ embedding elements, amongst which we have the $m-n$ simple roots $L_i - L_{i+1},~ 1\leq i \leq m-n$. In determining the embedding elements, one deduces that the roots $2L_i$ for $1 \leq i \leq m-n$ also satisfy. In fact, we have all the positive roots of an $\liesp_{2(m-n)}$. We see that we already reach the nilpotent orbits $\iota(H_{\liesp_{2(m-n)}})$. However, if $m-n +1$ is odd, we can form the set of simple roots $\{L_1 - L_2, \ldots, L_{m-n} - L_{m-n+1}\}$ and this corresponds to a chunk $\{2 \frac{m-n}{2} +1,2 \frac{m-n}{2} +1\}$, leading to an additional partition $[(m-n+1)^2,1^{2(n-1)}]$.

\begin{lemma}\lemlabel{Jordan2}
Let $\left( \begin{array}{cc} A & B \\ C & D \end{array}\right)$ have Jordan normal form $J$ then the Jordan normal form of
$$ J( \left( \begin{array}{cc|cc}
A & 0 & B & 0\\
0 & 0 & 0 & 0\\
\hline
C & 0 & D & 0\\
0 & 0 & 0 & 0 
\end{array}\right)) = \left( \begin{array}{c|c}
0 &  0\\
\hline
0 & J
\end{array}\right), $$
and
$$J(\left( \begin{array}{c|cc}
0 & 0 & 0\\
\hline
0 & A & B\\
0 & C & D
\end{array}\right) )= \left( \begin{array}{c|c}
0 &  0\\
\hline
0 & J
\end{array}\right),$$
where we have the same behavior of 0 as in \remref{zero}
\end{lemma}

\begin{theorem}
Let $\liesp_{2n} \subset \liesp_{2m}$, then the nilpotent orbits generated by the embedding elements are 
$$\begin{array}{cl}
H_{\liesp_{2m-2n}} & {\rm if~} m-n {\rm~is ~odd,}\\
H_{\liesp_{2m-2n}} \cup \mc{O}_{[(m-n+1)^2,1^{2(n-1)}]} & {\rm if~} m-n {\rm~is ~even.}
\end{array}$$
\end{theorem}
\begin{proof}
Follows by the above discussion and \lemref{Jordan2}.
\end{proof}

For type $B$ the result is less straightforward. We know by \cite[Theorem 5.1.2]{CoMc} that the nilpotent orbits of $\so_{2m+1}$ are in one-to-one correspondence with the set of partitions of $2m+1$ in which even parts occur with even multiplicity. The root system equals the root system of $\liesp_{2m}$ with the $2L_i$ replaced by $L_i$. Hence, an inclusion $\so_{2n+1} \subset \so_{2m+1}$ can be depicted by \eqref{rootsystem} with the $2L_m$ replaced by $L_m$. The difference with the symplectic case however, is that amongst the embedding elements we don't have the $L_i$ for $1 \leq i \leq m-n$. Hence we certainly do not reach all nilpotent orbits of an $\so_{2(m-n)+1}$. An element of $\so_{2m+1}$ has the form
$$\left( \begin{array}{ccc}
0 & u & v\\
-v^t & Z_1 & Z_2\\
-u^t & Z_3 & -Z_1^t
\end{array}\right), \quad u,v \in \mathbb{C}^m, Z_i \in M_m(\mathbb{C}), Z_2, Z_3 {\rm~skew-symmetric},$$
and the positive roots $\alpha$ for which the root vector $X_\alpha$ has non-zero $u$ or $v$ are exactly the $L_i, ~1 \leq i \leq m$. Since we do not have access to these guys, we get restrictions on the orbits we reach. Anyway, an orbit of $\so_{2n+1}$ given by a partition $[d_1^{n_1},\ldots, d_k^{n_k}]$ corresponds to the orbit $[d_1^{n_1},\ldots, d_k^{n_k}, 1^{2(m-n)}]$ in $\so_{2m+1}$. In \cite{CoMc} a recipe is given to construct the nilpotent element $X_{[d_1^{n_1},\ldots, d_k^{n_k}]}$. One has to break up the partition into chunks of three types: pairs $\{r,r\}$ of equal parts, pairs $\{2s +1,2t+1\}$ of unequal parts and one single odd part $\{2u+1\}$. One then associates positive roots to all three types of chunks and we observe that we only need some $L_i$ for a chunk of the last type $\{2u+1\}$. Moreover, if $u=0$ then we do not need such an $L_i$! Hence, we get access to embedded orbits of an $\so_{2(m-n)+1}$ represented by a partition ${\bf d}$ of $2(m-n)+1$ for which the unique chunk $\{2u+1\}$ has $u=0$. A moment's thought leads to the observation that a partition {\bf d} of $2m+1$ having at least one 1, can be broken up into chunks such that $\{2u+1\} = \{1\}$. So we arrive at

\begin{theorem}
Let $\so_{2n+1} \subset \so_{2m+1}$, then the nilpotent orbits generated by the embedding elements correspond to the partitions of $2(m-n)+1$ having at least one 1.
\end{theorem}
\begin{proof}
The above discussion shows that we can construct the $X_{{\bf d}}$ for ${\bf d}$ a partition of $2m +1$ corresponding to a partition of $2(m-n)+1$ having at least 1. Let $X$ be a nilpotent element constructed out of the embedding elements, then it is of the form 
 $$X = \left( \begin{array}{c|cc|cc}
0& 0 & 0 & 0 &0\\
\hline 
0 & A & 0 & B & 0\\
0 & 0 & 0 & 0 & 0\\
\hline
0&C & 0 & D & 0\\
0&0 & 0 & 0 & 0 
\end{array}\right)$$
and \lemref{Jordan2} entails that its Jordan form equals
$$\left( \begin{array}{c|c}
0 & 0 \\
\hline
0 & J(\left( \begin{array}{cc} A & B \\ C & D \end{array}\right))
\end{array}\right),$$
so the associated partition has at least one 1.
\end{proof}
 For example, for $\so_{5} \subset \so_{11}$, we have $2(m-n) + 1 = 7$ and so we don't reach the orbits $[3,2^2,1^4]$ and $[7,1^4]$.\\
 
 Finally, we discuss type $D$.

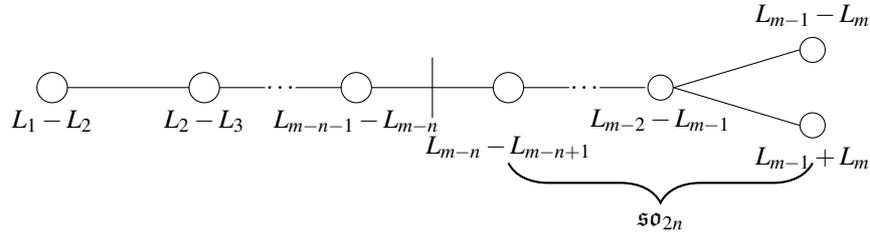
\begin{equation}
\begin{tikzpicture}
 \draw (0,0) circle (0.2cm) node[below=5pt] {$L_1 - L_2$};
  \draw ( 0.2,0)--( 2 -0.2,0);
 \draw (2,0) circle (0.2cm) node[below=5pt]{$L_2 - L_3$};
   \draw ( 2.2,0)--( 3 -0.2,0);
 \node at (3,0){$\ldots$};
 \draw ( 3.2,0)--(  4-0.2,0);
 \draw (4,0) circle (0.2cm) node[below=5pt] {$L_{m-n-1} - L_{m-n}$};
 \draw ( 4.2,0)--( 6 -0.2,0);
 \draw (5,.4) -- (5,-.4);
 \draw (6,0) circle (0.2cm) node[below=15pt] {$L_{m-n} - L_{m-n+1}$};
 \draw ( 6.2,0)--( 7 -0.2,0);
\node at (7,0) {$\ldots$};
\draw ( 7.2,0)--( 8 -0.2,0);
\node (five) at (8, 0) [circle,draw,label = below: $L_{m-2} - L_{m-1}$] {};
\node (six) at (10,.5) [circle,draw,label = above: $L_{m-1} - L_m$]{};
\node (seven) at (10,-.5) [circle,draw,label = below: $L_{m-1} + L_m$]{};

 \draw (five.east) -- (six.west);
  \draw (five.east) -- (seven.west);
 \draw [thick,decorate, decoration={brace,amplitude = 15pt,mirror}](6,-1) -- (10,-1) node[below=15pt,midway]{$\so_{2n}$};
\end{tikzpicture}
\end{equation}

With respect to this embedding, one deduces that amongst the embedding elements we have the roots $L_i \pm L_j,~ 1 \leq i < j \leq m-n$. These form the root system of an $\so_{2(m-n )}$, which has $(m-n)(m-n-1)$ positive roots. \propref{amountABCD} says that there are $m-n$ more embedding elements out there. Explicitly, these are the $L_i + L_{m-n+1}$ for $1 \leq i \leq m-n$. Springer and Steinberg showed that the nilpotent orbits in $\so_{2m}$ are parametrized by partitions on $2m$ in which even parts occur with even multiplicity, except that very even partitions (those with only even parts, each having even multiplicity) correspond to two orbits. In \cite{CoMc} a recipe is given to construct the nilpotent elements $X_{{\bf d}}$ corresponding to some partition ${\bf d}$ and we see that only for these very even partitions the roots $L_i + L_{i+1}$ are needed. Again, we automatically obtain the nilpotent orbits of an $\so_{2(m-n)}$. In some cases however, we also reach one of the two orbits associated to the very even partitions of $2m$! For such a very even partition ${\bf d}$ of $2m$ the recipe shows that we always need access to the root $L_{m-1} + L_m$, so we must have that $m-n+1 = m$ or $n=1$. Of course, for $2m$ to have a very even partition in the first place, $m$ must be even. Hence we are in the situation $\liesl_2 \subset \so_{4m'}$. We conclude
\begin{theorem}
Let $\so_{2n} \subset \so_{2m}$, then the nilpotent orbits generated by the embedding elements correspond to the inclusion of the Hasse subdiagram $\iota(H_{\so_{2(m-n)}})$. If in addition, $n=1$ and $m$ is even, i.e. $\liesl_2 \subset \so_{4m'}$ then we also reach one of the two orbits associated to a very even partition of $2m$.
\end{theorem}

\section*{Acknowledgement}
The author is grateful to Fred Van Oystaeyen for introducing him to the subject of glider representations and for collaborating with him on many different topics. The author also wishes to thank Jacques Alev for his suggestion to look at nilpotent orbits in particular and for the e-mail correspondence about Lie algebra related topics in general.

\end{document}